\documentclass[reqno,12pt]{amsart}
\usepackage{amssymb,amsmath}
\usepackage[dvips]{color}
\usepackage[dvips,final]{graphics}
\usepackage{epstopdf}
\usepackage{amscd}
\usepackage[latin1]{inputenc}
\theoremstyle{plain}
\newtheorem{teo}{Theorem}[section]      
\newtheorem{prop}[teo]{Proposition}    
\newtheorem{cor}[teo]{Corollary}       
\newtheorem{lem}[teo]{Lemma}            

\theoremstyle{definition}              
\newtheorem{defin}{Definition}[section]  
\theoremstyle{remark}                   
\newtheorem{oss}[teo]{Remark}

\newcommand{\re}{\mathbb{R}}

\numberwithin{equation}{section}
\newcommand{\usup}{\overline{u}}
 \newcommand{\vsup}{\overline{v}}                    
\begin{document}

\title[Saddle-shaped solutions of bistable elliptic equations involving the half-Laplacian]{Saddle-shaped solutions of bistable elliptic equations involving the half-Laplacian}

\thanks{
The author was supported by grants
MTM2008-06349-C0301 (Spain), 2009SGR345 (Catalunya) and partially supported by University of Bologna (Italy), funds for selected
research topics.
}
\subjclass[2010]{Primary: 35J61, 35J20; Secondary: 35B40, 35B08.}
\keywords{Half-Laplacian, saddle-shaped solutions, asymptotic behaviour, stability properties.}
\author[Eleonora Cinti]{Eleonora Cinti}
\address{
Max Planck Institute for Mathematics in the Sciences\\
Inselstr. 22\\
04103 Leipzig (Germany)} \email{cinti@mis.mpg.de}

\maketitle
\begin{abstract}
We establish existence and qualitative properties of
saddle-shaped solutions of the elliptic fractional equation
$(-\Delta)^{1/2}u=f(u)$ in all the space $\re^{2m}$, where $f$ is
of bistable type. These solutions are odd with respect to the Simons cone and even with
respect to each coordinate.

More precisely, we prove the existence of a saddle-shaped solution in every even dimension $2m$, 
as well as its monotonicity properties, asymptotic behaviour, and instability in dimensions $2m=4$ and $2m=6$.

These results are relevant in connection with the analog for fractional equations of a
conjecture of De Giorgi on the 1-D symmetry of certain
solutions. Saddle-shaped solutions are the simplest candidates, besides 1-D solutions, 
to be global minimizers in high dimensions, a property not yet established.
\end{abstract}

\section{Introduction and results}
This paper concerns the study of saddle-shaped solutions of
elliptic equations with fractional diffusion of the form
\begin{equation}\label{eq1-saddle}
(-\Delta)^{1/2}u=f(u)\quad \mbox{in}\;\re^n,\end{equation} where
$n=2m$ is an even integer and $f$ is of bistable type.

The fractional powers of the Laplacian are the infinitesimal generators of
 L{\'e}vy stable processes and appear in anomalous diffusion phenomena in plasmas, flame propagation, chemical reaction in liquids and population dynamics.

Our interest in saddle-shaped solutions originates from the following
conjecture of De Giorgi.
Consider the Allen-Cahn equation
\begin{equation}\label{AC-saddle}
 -\Delta u=u-u^3\quad \mbox{in}\;\re^n,
\end{equation}
which models phase transitions.
In 1978 De Giorgi conjectured that the
level sets of every bounded solution of \eqref{AC-saddle}, which is monotone in one direction, must be hyperplanes, at least if $n\leq 8$.
That is, such solutions depend only on one Euclidian variable.

The
conjecture has been proven to be true in dimension $n=2$ by
Ghoussoub and Gui \cite{GG} and in dimension $n=3$ by Ambrosio and
Cabr\'e \cite{AC}. For $4\leq n\leq 8$, if $\partial_{x_n}u>0$,
and assuming the additional condition
$$\lim_{x_n \rightarrow \pm \infty}u(x',x_n)=\pm 1\quad \mbox{for all}\;x'\in \re^{n-1},$$ it has been
established by Savin \cite{S}. Recently a counterexample to the
conjecture for $n\geq 9$ has been found by del Pino, Kowalczyk
and Wei \cite{dPKW}.

For the fractional equation $(-\Delta)^su=f(u)$ in $\re^n$ with
$0<s<1$, the conjecture has been proven to be true when $n=2$ and
$s=1/2$ by Cabr\'e and Sol\`a-Morales \cite{C-SM}, and when $n=2$
and for every $0<s<1$ by Cabr\'e and Sire \cite{C-Si1}, and by Sire
and Valdinoci \cite{SV}. In two recent papers \cite{C-Cinti1,C-Cinti2}, Cabr\'e and the author prove the conjecture in
dimension $n=3$ for every power $1/2\leq s<1$.

Coming back to the classical Allen-Cahn equation, Savin \cite{S}, proved that if $n\leq 7$ then every global
minimizer of the equation $-\Delta u=u-u^3$ in $\re^n$ is
one-dimensional.
A natural question arises: is there a global
minimizer in $\re^8$ which is not one-dimensional? Saddle-shaped
solutions are the candidates to give a positive answer to this
question, which is still an open problem.

Moreover, by a result of Jerison and Monneau \cite{JM}, if one could prove that saddle-shaped
solutions are global minimizers in $\re^8$, one would have a
counterexample to the conjecture of De Giorgi in $\re^9$, in an
alternative way to that of \cite{dPKW}.

Saddle-shaped solutions are expected to have relevant
variational properties due to a well known connection between
nonlinear equations modeling phase transitions and the theory of
minimal surfaces. This connection also motivated the conjecture of
De Giorgi.

More precisely, the saddle-shaped solutions that we
consider are even with respect to the coordinate axes and odd with
respect to the Simons cone, which is defined as follows. For
$n=2m$ the Simons cone ${\mathcal C}$ is given by:
$$\mathcal{C}=\{x\in
\re^{2m}:x_1^2+...+x_m^2=x_{m+1}^2+...+x_{2m}^2\}.$$ We recall
that the Simons cone has zero mean curvature at every point $x\in
\mathcal C \setminus \{0\}$, in every dimension $2m\geq 2$.
Moreover in dimensions $2m \geq 8$ it is a minimizer of the area
functional, that is, it is a minimal cone (in the variational
sense).

We define two new variables $$s=\sqrt{x_1^2+\dots + x_m^2} \quad
\mbox{ and }\quad t=\sqrt{x_{m+1}^2+\dots + x_{2m}^2},$$ for which
the Simons cone becomes ${\mathcal C}=\{s=t\}$.

We now introduce the notion of saddle-shaped solution. These solutions
depend only on $s$ and $t$, and are odd with respect to the Simons
cone.
\begin{defin}\label{def-saddle} Let $u$ be a bounded solution of
$(-\Delta)^{1/2} u=f(u)$ in $\re^{2m}$, where $f\in C^1$
is odd. We say that $u:\re^{2m}\rightarrow\re$ is a  {\it saddle-shaped}
 (or simply \textit{saddle}) solution if
\renewcommand{\labelenumi}{$($\alph{enumi}$)$}
\begin{enumerate}
\item $u$ depends only on the variables $s$ and $t$. We write
$u=u(s,t)$; \item $u>0$ for $s>t$; \item $u(s,t)=-u(t,s)$.
\end{enumerate}
\end{defin}

\begin{oss}
 If $u$ is a saddle solution then, in particular, $u=0$ on the
Simons cone ${\mathcal C}=\{s=t\}$. In other words, ${\mathcal C}$
is the zero level set of $u$.
\end{oss}

Saddle solutions for the classical equation $-\Delta u=f(u)$ were
first studied by Dang, Fife, and Peletier in \cite{DFP} in
dimension $2m=2$ for $f$ odd, bistable and $f(u)/u$ decreasing for
$u \in (0,1)$. They proved the existence and uniqueness of
saddle-shaped solutions and established monotonicity properties
and the asymptotic behaviour. The instability property of saddle
solutions in dimension $2m=2$ was studied by Schatzman \cite{Sc}.
In two recent works \cite{CT, CT2}, Cabr\'e and Terra proved the existence of
saddle-shaped solutions for the equation $-\Delta u=f(u)$ in
$\re^{2m}$, where $f$ is of bistable type, in every even dimension
$2m$. Moreover they established some qualitative properties of
these solutions, such as monotonicity properties, asymptotic behaviour,
and also instability in dimensions $2m=4$ and $2m=6$.

In this work, we establish existence and qualitative properties
of saddle-shaped solutions for the bistable fractional equation
\eqref{eq1-saddle}.

To study the nonlocal problem (\ref{eq1-saddle}) we will realize it as a
local problem in $\re^{n+1}_+$ with a nonlinear Neumann condition
on $\partial \re_{+}^{n+1}=\re^n$. More precisely, if $u=u(x)$ is
a function defined on $\re^n$, we consider its harmonic extension
$v=v(x,\lambda)$ in $\re^{n+1}_+=\re^n\times(0,+\infty)$. It is
well known (see \cite{C-SM,CS}) that $u$ is a solution of
(\ref{eq1-saddle}) if and only if $v$ satisfies
\begin{equation}\label{eq2-saddle}
\begin{cases}
\Delta v=0& \text{in}\; \re_{+}^{n+1},\\
- \partial_{ \lambda}v=f(v)& \text{on}\; \re^{n}=\partial
\re_{+}^{n+1}.
\end{cases}
\end{equation}

Problem (\ref{eq2-saddle}), associated to the nonlocal equation
\eqref{eq1-saddle}, allows to introduce the notions of \textit{energy}, \textit{stability},
 and \textit{global minimality} for a solution $u$ of problem (\ref{eq1-saddle}).

Let $\Omega \subset \re^{n+1}_+$ be a bounded domain. We denote by
$$\widetilde{B}_r^+=\{(x,\lambda)\in \re^{2m+1}:\lambda>0,\:|(x,\lambda)|<r\}$$ and by
$\widetilde{B}^+_r(x,\lambda)=(x,\lambda)+\widetilde{B}_r^+$.

We define the following subset of $\partial \Omega$:
\begin{equation}\label{bordo0}
\partial^0\Omega:=\{(x,0)\in
\re^{n+1}_+:\widetilde{B}^+_{\varepsilon}(x,0)\subset \Omega
\;\mbox{for some}\;\varepsilon>0\}\end{equation} and
\begin{equation}\label{bordo+}
\partial^+\Omega:=\overline{\partial\Omega \cap\re^{n+1}_+}.
\end{equation}

Given a $C^{1,\alpha}$ nonlinearity
$f:\re\rightarrow \re$, for some $0<\alpha<1$, define
$$G(u)=\int_u^1 f.$$
We have that $G\in C^2(\re)$ and $G'=-f$.

Let $v$ be a $C^1(\overline{\Omega})$ function with $|v|\leq 1$.
We consider the energy functional
\begin{equation}\label{energia-saddle}
{\mathcal E}_{\Omega}(v)=\int_{\Omega}\frac{1}{2}|\nabla v|^2 +
\int_{\partial^0\Omega}G(v). \end{equation}
Observe that the potential energy is computed only on the boundary
$\partial^0\Omega\subset \partial\re_+^{n+1}$. This is a quite different
situation from the one of interior reactions.

We start by recalling that problem ($\ref{eq2-saddle}$) can be viewed as
the Euler-Lagrange equation associated to the energy functional
${\mathcal E}$.

%

\begin{defin}\label{stable}
\textit{a)}  We say that a bounded solution $v$ of
(\ref{eq2-saddle}) is {\it stable} if the second variation of energy
$\delta^2{\mathcal E}/\delta^2\xi$, with respect to  perturbations
$\xi$ compactly supported in $\overline{\re_+^{n+1}}$, is
nonnegative. That is, if
\begin{equation}\label{Q_v}
Q_v(\xi):=\int_{{\re}_+^{n+1}}
|\nabla\xi|^2-\int_{\partial\re_+^{n+1}}f'(v)\xi^2\geq
0\end{equation} for every $\xi\in
C_0^\infty(\overline{\re_+^{n+1}})$.

We say that $v$ is {\it unstable} if and only if $v$ is not
stable.\\
\textit{b)}  We say that a bounded solution $u$ of
(\ref{eq1-saddle}) in $\re^{2m}$ is {\it stable} ({\it unstable}) if its
harmonic extension $v$ is a stable (unstable) solution for the
problem (\ref{eq2-saddle}).
\end{defin}

 Another important
notion related to the energy functional ${\mathcal E}$ is the one
of global minimality.

\begin{defin}\label{minimizer}
\textit{a)}  We say that a bounded $C^1(\overline{\re_+^{n+1}})$ function $v$ in $\re^{n+1}_+$ is a
{\it global minimizer} of $(\ref{eq2-saddle})$ if
$${\mathcal E}_{\Omega}(v)\leq {\mathcal E}_{\Omega}(v+\xi),$$ for
every bounded domain $\Omega\subset \overline{\re_+^{n+1}}$ and
every $C^{\infty}$ function $\xi$ with compact support in
$\Omega\cup \partial^0\Omega$.\\
\textit{b)}  We say that a bounded $C^1$ function $u$ in $\re^{n}$ is a {\it
global minimizer} of $(\ref{eq1-saddle})$ if its harmonic extension $v$
is a global minimizer of (\ref{eq2-saddle}).
\end{defin}
Observe that the perturbations $\xi$ do not need to vanish on
$\partial^0\Omega$, in contrast from interior reactions.

In some references, global minimizers are called ``local
minimizers'', where local stands for the fact that the energy is
computed in bounded domains. Clearly, every global minimizer is a
stable solution.

In what follows we will assume some or all of the following properties on $f$: 
\begin{equation}\label{h1}
f \; \text{ is odd};\end{equation}
\begin{equation}\label{h2}
 G\geq 0=G(\pm 1)\; {\rm in\, } \re, \,{\rm and }\, G>0\; {\rm in }\, (-1,1);\end{equation}
\begin{equation}\label{h3}  f' \; \text{ is decreasing in } (0,1).\end{equation}

Note that, if \eqref{h1} and \eqref{h2} hold, then $f(0)=f(\pm
1)=0.$ Conversely, if $f$ is odd in $\re$, positive with $f'$
decreasing in $(0,1)$ and negative in $(1,\infty)$ then $f$
satisfies \eqref{h1}, \eqref{h2} and \eqref{h3}. Hence, the
nonlinearities $f$ that we consider are of ``balanced bistable
type", while the potentials $G$ are of ``double well type". Our
three assumptions \eqref{h1}, \eqref{h2}, \eqref{h3} are satisfied
for the scalar Allen-Cahn type equation
\begin{equation}\label{GL}
(-\Delta)^{1/2} u= u-u^3.
\end{equation}
In this case we have that $G(u)=(1/4)(1-u^2)^2$ and \eqref{h1},
\eqref{h2}, \eqref{h3} hold. The three hypothesis also hold for
the Peierls-Nabarro problem
\begin{equation}\label{PN}(-\Delta)^{1/2} u=\sin (\pi u),\end{equation} for
which $G(u)=(1/\pi)(1+\cos (\pi u))$.

By a result of Cabr\'e and Sol\`a-Morales \cite{C-SM}, assumption
\eqref{h2} on $G$ guarantees the existence of an increasing solution, from
$-1$ to $1$, of \eqref{eq1-saddle} in $\re$. We call these solutions \textit{layer} solutions. In addition, such
an increasing solution is unique up to translations.

The following is the precise result established in \cite{C-SM}.
\begin{teo}\label{esistenza-layer}{\textbf{\it{(\cite{C-SM})}}}
Let $f$ be any $C^{1,\alpha}$ function with
$0<\alpha<1$ and $G'=-f$. Then:
\begin{itemize}
\item There exists an increasing solution $u_0:\re\rightarrow (-1,1)$ of
$(-\Delta)^{1/2}u_0=f(u_0)$ in $\re$ (that is, a layer solution $u_0$)
if and
only if
$$G'(-1)=G'(1)=0,\;\;\mbox{and}\;\;G>G(-1)=G(1)\;\;\mbox{in}\;(-1,1).$$
\item If $f'(\pm 1)<0$, then a layer solution of (\ref{eq2-saddle}) is
unique up to translations. \item If $f$ is odd
and $f'(\pm 1)<0$, then every layer solution of (\ref{eq2-saddle}) is odd
in $x$ with respect to some half-axis. That is,
$u(x+b)=-u(-x+b)$ for some $b\in
\re$.\end{itemize}
\end{teo}
Normalizing the layer solution to vanishing at the origin, we call
it $u_0$ and its harmonic extension in the half-plane $v_0$. Thus we have
\begin{equation}\label{eq-layer}
\begin{cases}
u_0: \re \rightarrow (-1,1)\\
u_0(0)=0,\quad u_0'>0\\
(-\Delta)^{1/2}u_0=f(u_0)\quad \mbox{in}\:\:\re.
\end{cases}
\end{equation}
The monotone bounded solution $u_0$
of the Peierls-Nabarro problem \eqref{PN} in $\re$ is explicit. Calling $v_0$ its harmonic extension in $\re^2_+$ we have that
$$v_0(x,\lambda)=\frac{2}{\pi}\arctan {\frac{x}{\lambda + 1/\pi}}.$$

In the following theorem, we establish the existence of a
saddle-shaped solution for problem \eqref{eq1-saddle} in every even
dimension $n=2m$.
We use the following notations:
$$\mathcal{O}:=\{x\in \re^{2m}:s>t\}\subset \re^{2m}$$
$$\widetilde{\mathcal{O}}:=\{(x,\lambda)\in \re^{2m+1}_+:x\in
\mathcal O\}\subset \re^{2m+1}_+$$
Note that
$$
\partial{\mathcal O}={\mathcal C}.
$$
We define the cylinder $$C_{R,L}=B_R\times (0,L),$$ where $B_R$ is the open ball in $\re^{2m}$ centered at the origin and of radius $R$.
\begin{teo}\label{existence}
For every dimension $2m\geq 2$ and every nonlinearity $f$
satisfying \eqref{h1} and \eqref{h2}, there exists a saddle
solution $u$ of $(-\Delta)^{1/2} u= f(u)$ in $\re^{2m}$,
such that $|u|<1$ in $\re^{2m}$.

Let $v$ be the harmonic extension of the saddle solution $u$ in $\re^{2m+1}_+$. If in addition $f$ satisfies \eqref{h3}, then the second variation of the energy $Q_v(\xi)$ at v, as defined in \eqref{Q_v}, is nonnegative for all function $\xi\in C^1(\overline{\re^{2m+1}_+})$ with compact support in $\overline{\re^{2m+1}_+}$ and vanishing on $\mathcal C\times [0,+\infty)$.
\end{teo}

We prove the existence of a saddle solution $u$ for problem
(\ref{eq1-saddle}), by proving the existence of a solution $v$ for
problem (\ref{eq2-saddle}), with the following properties:
\begin{enumerate}
\item $v$ depends only on the variables $s,\: t$ and $\lambda$. We
write $v=v(s,t,\lambda)$; \item $v>0$ for $s>t$; \item
$v(s,t,\lambda)=-v(t,s,\lambda)$.
\end{enumerate}

Using a variational technique we construct a solution $v$ for the following problem
$$\begin{cases}
\Delta v=0&\mbox{in} \; \widetilde{\mathcal O}\\
v>0&\mbox{in} \; \widetilde{{\mathcal O}}\\
v=0& \mbox{on}\; \mathcal C\times [0,+\infty)\\
-\partial_{\lambda} v=f(v)& \mbox{on}\; \mathcal O\times\{\lambda=0\}.\end{cases}$$
Then,
since $f$ is odd, by odd reflection with respect to
$\mathcal C\times [0,+\infty)$ we obtain a solution $v$ in the
whole space which satisfies properties (1), (2), (3). Clearly the
function $u(x)=v(x,0)$ is a saddle solution for problem
(\ref{eq1-saddle}). 

To prove this existence result, we will use the following non-sharp energy estimate for $v$.
Given $1/2\leq \gamma <1$, there exists $\varepsilon=\varepsilon(\gamma)>0$ such that
\begin{equation}\label{en-saddle}\mathcal E_{C_{S,S^{\gamma}}}(v)\leq CS^{2m-\varepsilon}.\end{equation}

In Theorem 1.7 of \cite{C-Cinti1}, Cabr\'e and the author establish the following sharp
energy estimates for saddle-shaped solutions,
$$\mathcal E_{C_{S,S}}(v)\leq CS^{2m-1}\log S.$$
Here,
\eqref{en-saddle} is not sharp, but it is enough to prove the
existence of a saddle solution.

For solutions of problem (\ref{eq2-saddle}) depending only on the coordinates $s,\;t$ and $\lambda$,
 problem (\ref{eq2-saddle}) becomes
\begin{equation}\label{eqst}
\begin{cases}
\displaystyle -(v_{ss}+v_{tt}+v_{\lambda\lambda})-(m-1)\left
(\frac{v_s}{s}+\frac{v_t}{t}\right )=0,& \mbox{in}\:\:\re^{2m+1}_+\\
\displaystyle -\partial_{\lambda}v=f(v) & \mbox{on}\:\:\partial\re^{2m+1}_+, \end{cases}
\end{equation}  while the energy functional
becomes
\begin{equation}\label{enerst}
{\mathcal E}(v,\Omega)=c_m \left\{\int_\Omega s^{m-1}t^{m-1}
\frac{1}{2} (v_s^2+v_t^2+v_\lambda^2)ds dt
d\lambda+\int_{\partial^0 \Omega}s^{m-1}t^{m-1}G(v) ds dt\right\},
\end{equation}
where $c_m$ is a positive constant depending only on $m$---here we
 have
 assumed that $\Omega\subset\re^{2m+1}$ is radially symmetric
in the first $m$ variables and also in the last $m$ variables, and
we have abused notation by identifying $\Omega$ with its
projection in the $(s,t,\lambda)$ variables.

In section 5, we prove the existence and
monotonicity properties of a maximal saddle solution.

To establish these results, we need to introduce a new nonlocal
operator $D_{H,\varphi}$, which is the square root of the Laplacian for functions
defined in domains $H \subset \re^n$ which do not vanish on $\partial
H$. We introduce this operator and we establish maximum
principles for it, in section 4.

We define the new variables
\begin{equation}\label{defyz}
\begin{cases}
\displaystyle y  =   \frac{s+t}{\sqrt{2}} \vspace{1em}\\
\displaystyle z =  \frac{s-t}{\sqrt{2}}.
\end{cases}
\end{equation}
Note that $|z|\leq y$ and that we may write the Simons cone as
${\mathcal C}=\{z=0\}$.

The following theorem concerns the existence and
monotonicity properties of a maximal saddle solution.

\begin{teo}\label{maximal-u}
  Let $f$ satisfy conditions \eqref{h1}, \eqref{h2}, and \eqref{h3}.

Then, there exists a saddle solution $\usup$ of $(-\Delta)^{1/2}\usup=f(\usup)$ in $\re^{2m}$, with $|u|< 1$, which is maximal in the following sense. For every solution $u$ of $(-\Delta)^{1/2}u=f(u)$ in $\re^{2m}$, vanishing on the Simons cone and such that $u$ has the same sign as $s-t$, we have
$$0<u\leq\usup\quad \mbox{in}\;\;\mathcal O.$$
As a consequence, we also have
$$0\leq|u|\leq|\usup|\quad \mbox{in}\;\;\re^{2m}.$$

In addition, if $\vsup$ is the harmonic extension of $\usup$ in $\re^{2m+1}_+$,
then $\vsup$ satisfies:
\begin{enumerate}
 \item [(a)] $\partial_s \vsup \geq 0$ in $\overline{\re^{2m+1}_+}$. Furthermore $\partial_s \vsup >0$ in $\overline{\re^{2m+1}_+}\setminus\{s=0\}$      and $\partial_s\vsup=0$ in $\{s=0\}$;
\item [(b)] $\partial_t \vsup \leq 0$ in $\overline{\re^{2m+1}_+}$. Furthermore $\partial_t \vsup <0$ in $\overline{\re^{2m+1}_+}\setminus\{t=0\}$      and $\partial_t\vsup=0$ in $\{t=0\}$;
\item [(c)] $\partial_z \vsup > 0$ in $\overline{\re^{2m+1}_+}\setminus\{0\}$;
\item [(d)] $\partial_y \vsup > 0$ in $\{s>t\}\times [0,+\infty)$.

As a consequence, for every direction $\partial_{\eta}=\alpha\partial_y-\beta \partial_t$, with $\alpha$ and $\beta$ positive constants, $\partial_{\eta}\vsup>0$ in $\{s>t>0\}\times [0,+\infty)$.\end{enumerate}
\end{teo}

Theorem \ref{maximal-u} above is the analog of Theorem 1.7 in \cite{CT2} for reactions in the interior. In \cite{CT2} two important ingredients in the proof of the 
existence and monotonicity properties of the maximal saddle solution are the following.
Let $u^{(1)}$ be a saddle solution of $-\Delta u^{(1)}= f(u^{(1)})$ in $\re^{2m}$, with $f$ bistable, and let $u^{(1)}_0$ be the layer 
solution in dimension $n=1$ of $(-u^{(1)}_0)''=f(u^{(1)}_0)$ (whose existence is guaranteed by hypothesis \eqref{h2} on $f$).
Then
\begin{enumerate}
 \item [i)] $u^{(1)}_0(|s-t|/\sqrt 2)$ is a supersolution of $-\Delta u^{(1)}= f(u^{(1)})$ in $\mathcal O$;
\item [ii)]
\begin{equation}\label{bound-mo}
|u^{(1)}(x)|\leq \left|u^{(1)}_0\left(d(x,\mathcal C)\right)\right|=\left|u^{(1)}_0\left(\frac{|s-t|}{\sqrt 2}\right)\right|\quad \mbox{for every}\:\:x\in \re^{2m},
\end{equation}
where $d(\cdot,\mathcal C)$ denotes the distance to the Simons cone.
\end{enumerate}
The following proposition  establishes the analog for boundary reactions of point i) above.
\begin{prop}\label{super-ext}
Let $f$ satisfy hypothesis \eqref{h1}, \eqref{h2}, \eqref{h3}. Let $u_0$ be the layer solution, vanishing at the origin, of problem \eqref{eq1-saddle} 
in $\re$ and let $v_0$ be its harmonic extension in $\re^{2}_+$.

Then, the function $\displaystyle v_0(z,\lambda)=v_0\left(\frac{s-t}{\sqrt 2},\lambda\right)$ satisfies
\begin{equation*}
\begin{cases}
 -\Delta v_0\geq 0 &\mbox{in} \:\:\widetilde{\mathcal O}\\
-\partial_{\lambda}v_0\geq f(v_0)   &\mbox{on} \:\:{\mathcal O}\times \{0\}.
\end{cases}
\end{equation*}
\end{prop}

Concerning point ii) above, estimate \eqref{bound-mo} follows by an important gradient bound of 
Modica \cite{Mo} for the classical equation 
$-\Delta u= f(u)$ in $\re^n$. 

In the fractional case Cabr\'e and Sol\`a-Morales \cite{C-SM} and Cabr\'e and Sire \cite{C-Si1} established a non-local version of the Modica estimate in dimension $n=1$, 
the analog estimate for dimentsions $n>1$ is still an open problem.
Therefore, we are not able to deduce the analog of \eqref{bound-mo} for solutions of the equation $(-\Delta )^{1/2}u=f(u)$ in $\re^{2m}$.
For this reason, to give an upper barrier for saddle solutions, that at the same time is a supersolution, we consider the function $\min\{Kv_0(|s-t|/\sqrt 2,\lambda),1\}$ 
where $K\geq 1$ is a large constant depending only on $n$, $||u||_{\infty}$, and $f$. Proposition \ref{super-ext} implies that this function is a 
supersolution in $\widetilde{\mathcal O}$. Moreover, we will show that there exists $K\geq 1$, depending only on $n$, $||u||_{\infty}$, and $f$, 
such that if $v$ is a bounded solution of problem \eqref{eq2-saddle}, vanishing on $\mathcal C\times [0,+\infty)$, then
\begin{equation}\label{bound-super}|v(x,\lambda)|\leq  \min\{Kv_0(|s-t|/\sqrt 2,\lambda),1\},\quad \mbox{for every}\:\:(x,\lambda)\in \overline{\re_+^{2m+1}}.\end{equation}
Estimate \eqref{bound-super} follows by regularity results established in \cite{C-SM}. 

In section 6, we prove the following theorem concerning the asymptotic behaviour at infinity for a class of solutions which contains saddle-shaped solutions.

\begin{teo}\label{asym}
Let $f$ satisfy conditions  \eqref{h1}, \eqref{h2}, and \eqref{h3}, and
let $u$ be a bounded solution of $(-\Delta)^{1/2} u=f(u)$
in $\re^{2m}$ such that $u\equiv 0$ on~${\mathcal C}$, $u>0$ in
${\mathcal O}=\{s>t\}$ and $u$ is odd with respect to ${\mathcal
C}$.

Then, denoting $U(x):=u_0((s-t)/\sqrt{2})=u_0(z)$ we have,
\begin{equation}
u(x)-U(x)\rightarrow 0 \quad\mbox{ and } \quad\nabla u(x)-\nabla
 U(x)\rightarrow 0,
\end{equation} uniformly as $|x|\rightarrow\infty.$ That is,
\begin{equation}\label{unif}
||{u-U}||_{L^{\infty}(\re^{2m}\setminus B_R)}+||\nabla
  u- \nabla U||_{L^{\infty}(\re^{2m}\setminus B_R)}\rightarrow 0\, \text{
  as }\, R\rightarrow\infty.
\end{equation}
\end{teo}

Our proof of Theorem \ref{asym} follows the one given by Cabr\'e and Terra in \cite{CT2}, and uses a compactness argument based on translations of the solutions, combined with two crucial Liouville-type results for nonlinear equations in the half-space and in a quarter of space.

Finally, in section 7 we establish that saddle-shaped solutions are unstable in dimension $2m=4$ and $2m=6$.

\begin{teo}\label{uns6} Let $f$ satisfy conditions \eqref{h1}, \eqref{h2}, \eqref{h3}.
Then, every bounded solution $u$ of $(-\Delta)^{1/2} u=f(u)$
in $\re^{2m}$ such that $u=0$ on the Simons cone ${\mathcal
C}=\{s=t\}$ and $u$ has the same sign as $s-t$, is unstable in
dimension ${2m=4}$ and $2m=6$.\end{teo}

Instability in dimension $2m=2$ follows by a result of Cabr\'e and Sol\`a Morales \cite{C-SM} which asserts that every stable solution of \eqref{eq1-saddle} in dimension $n=2$ is one-dimensional. This is the analog of the conjecture of De Giorgi in dimension $n=2$ for the half-Laplacian.

In \cite{CT}, Cabr\'e and Terra proved instability in dimension $2m=4$ for saddle-shaped solutions of the classical equation $-\Delta u=f(u)$ in $\re^4$. A crucial ingredient in the proof of this result is the pointwise estimate \eqref{bound-mo}.

However, in dimension $2m=6$, this estimate is not enough to prove instability and thus Cabr\'e and Terra used a more precise argument, based on some monotonicity properties and asymptotic behaviour of a maximal saddle solution.

Since, as said before, we cannot prove the analog of \eqref{bound-mo} for solutions of the equation $(-\Delta )^{1/2}u=f(u)$,
 here we follow the argument introduced by Cabr\'e and Terra in dimension $2m=6$, both for the case $2m=4$ and $2m=6$.

Using this approach, the crucial ingredients in the proof of Theorem \ref{uns6} are:
\begin{itemize}
\item [i)] the equation satisfied by $\vsup_z$, where $\vsup$ is the harmonic extension of the maximal saddle solution $\usup$ in $\re^{2m+1}_+$;
\item [ii)] a monotonicity property of $\vsup$;
\item [iii)] the asymptotic behaviour at infinity of $\vsup$.
\end{itemize}

The paper is organized as follows:
\begin{itemize}
\item In section 2 we prove Theorem \ref{existence} concerning the existence of a saddle solution for the equation \eqref{eq1-saddle} in every dimension $2m$.
\item In section 3, we give a supersolution and a subsolution for the square root of the Laplacian in a domain $H\subset \re^n$. In particular we prove Proposition \ref{super-ext}.
\item In section 4, we introduce the operator $D_{H,\varphi}$ and we establish maximum principles for it.
\item In section 5, we prove the existence of a maximal saddle solution $\usup$  and its monotonicity properties (Theorem \ref{maximal-u}).
\item In section 6, we prove Theorem \ref{asym}, concerning the asymptotic behaviour of saddle solutions.
\item In section 7, we prove Theorem \ref{uns6} about the instability of saddle solutions in dimensions $2m=4$ and $2m=6$.
\end{itemize}

\section{Existence of a saddle solution in $\re^{2m}$}

In this section we prove the existence of a saddle solution $u$
for problem (\ref{eq1-saddle}), by proving the existence of a solution
$v$ for problem (\ref{eq2-saddle}) with the following properties:
\begin{enumerate}
\item $v$ depends only on the variables $s,\: t$ and $\lambda$. We
write $v=v(s,t,\lambda)$; \item $v>0$ for $s>t$; \item
$v(s,t,\lambda)=-v(t,s,\lambda)$.
\end{enumerate}

We recall that we have defined the sets:
$$\mathcal{O}=\{x\in \re^{2m}:s>t\}\subset \re^{2m},\quad
\widetilde{\mathcal{O}}=\{(x,\lambda)\in \re^{2m+1}_+:x\in
\mathcal O\}\subset \re^{2m+1}_+.$$
Let $B_R$ be the open ball in $\re^{2m}$ centered at the origin
and of radius $R$. We will consider the open bounded
sets
$$
{\mathcal O}_R:={\mathcal O}\cap B_R=\{s>t,
 |x|^2=s^2+t^2<R^2\}\subset \re^{2m}.
$$
$$\widetilde{\mathcal{O}}_{R,L}:=(\mathcal O \cap B_R)\times (0,L)=\{(x,\lambda)\in
\re^{2m+1}_+:s>t,\:|x|^2=s^2+t^2<R^2, \lambda<L\}.$$ Note that
$$
\partial {\mathcal O}_R=({\mathcal C}\cap \overline{B}_R)\cup
(\partial{B_R}\cap {\mathcal O}).
$$

Before giving the proof of Theorem \ref{existence}, we recall some results established in \cite{C-SM} concerning the regularity of weak solutions of problem \eqref{eq2-saddle}. Cabr\'e and Sol\`a-Morales \cite{C-SM} proved that every bounded weak solution $v$ of problem \eqref{eq2-saddle} with $f\in C^{1,\alpha}$, satisfies $v\in C^{1,\alpha}$, for all $0<\alpha<1$. This result was deduced using the auxiliary function
$$w(x,\lambda)=\int_0^{\lambda}v(x,t)dt,$$
which is a solution of the Dirichlet problem
$$\begin{cases}
   -\Delta w=f(v(x,0)) &\mbox{in}\:\: \re^{2m+1}_+\\
w(x,0)=0 &\mbox{on}\:\: \partial \re^{2m+1}_+
  \end{cases}.
$$
Applying standard regularity results for the Dirichlet problem above, they deduce regularity for the solution $v$ of problem \eqref{eq2-saddle}.
Moreover, using standard elliptic estimates for bounded harmonic functions, we have that the following gradient bound for $v$ holds:
\begin{equation}\label{grad}
 |\nabla v(x,\lambda)|\leq \frac{C}{1+\lambda}\quad \mbox{for every}\;\: (x,\lambda)\in \overline{\re^{2m+1}_+}.
\end{equation}
We define now the sets
$$\widetilde L^2(\widetilde{{\mathcal O}}_{R,L})=\{v\in
L^2(\widetilde{{\mathcal O}}_{R,L}) :
v=v(s,t,\lambda) \text{ a.e.}\}
$$ and
$$
\widetilde{H}_{0}^1(\widetilde{{\mathcal O}}_{R,L})=\{v\in
H^1(\widetilde{{\mathcal O}}_{R,L}) : v\equiv 0
\:\:\mbox{on}\:\:\partial^+\widetilde{{\mathcal O}}_{R,L},\:
v=v(s,t,\lambda) \text{ a.e.}\}.
$$
They are, respectively, the set of $L^2$ functions in the bounded open set $\widetilde{{\mathcal
O}}_{R,L}$ which depend only on $s$, $t$,  and $\lambda$, and the set of $H^1$ functions in the bounded open set $\widetilde{{\mathcal
O}}_{R,L}$ which depend only on $s$, $t$ and $\lambda$ and which
vanish on the positive boundary $\partial^+\widetilde{{\mathcal
O}}_{R,L}$ in the weak sense.

We recall that the inclusion $\widetilde{H}_{0}^1(\widetilde{{\mathcal O}}_{R,L}) \subset\subset
L^2(\widetilde{{\mathcal O}}_{R,L})$ is compact (see \cite{C-SM}). Indeed, let $v \in \widetilde{H}_{0}^1(\widetilde{{\mathcal O}}_{R,L})$. Since $v\equiv 0$ on $\partial^+\widetilde{{\mathcal O}}_{R,L}$, 
we can extend $v$ to be identically $0$ in $\re^{2m+1}_+\setminus \widetilde{{\mathcal O}}_{R,L}$, so that $v\in \widetilde{H}^1(\re^{2m+1}_+)=\{v\in
H^1(\re^{2m+1}_+) :
v=v(s,t,\lambda) \text{ a.e.}\}$.
We have
\begin{eqnarray*}
\int_{\partial^0\widetilde{{\mathcal O}}_{R,L}}|v(x,0)|^2 dx =-\int_{\re^{n+1}_+}\partial_{\lambda}(|v|^2)=-2\int_{\re^{n+1}_+}v\partial_{\lambda}v\leq C||v||_{\widetilde L^2(\widetilde{{\mathcal O}}_{R,L})}||v||_{\widetilde H^1(\widetilde{{\mathcal O}}_{R,L})}.
\end{eqnarray*}
Now, the compactness of the inclusion, follows from the fact that since $v\equiv 0$ on $\partial^+\widetilde{{\mathcal O}}_{R,L}$ a.e., then $\widetilde{H}_{0}^1(\widetilde{{\mathcal O}}_{R,L}) \subset\subset
\widetilde L^2(\widetilde{{\mathcal O}}_{R,L})$ is compact (to see this it is enough to extend $v$ to be identically zero in a $A\setminus \widetilde{{\mathcal O}}_{R,L}$, where $A\subset \re^{n+1}_+$ is a Lipschitz set containing $\widetilde{{\mathcal O}}_{R,L}$).

We can now give the proof of Theorem \ref{existence}.

\begin{proof}[Proof of Theorem \ref{existence}]
As already mentioned, we prove the existence of a solution $v$ for
the problem (\ref{eq2-saddle}) such that $v=v(s,t,\lambda)$ and
$v(s,t,\lambda)=-v(-t,s,\lambda)$. The space $\widetilde{H}_{0}^1(\widetilde{{\mathcal O}}_{R,L}) $, defined above,
is a weakly closed subspace of $H^1(\widetilde{{\mathcal O}}_{R,L})$.

Consider the energy functional in $\widetilde{{\mathcal
O}}_{R,L}$,
$$
{\mathcal E}_{\widetilde{{\mathcal
O}}_{R,L}}(v)=\int_{\widetilde{{\mathcal O}}_{R,L}}
\frac{1}{2}|\nabla v|^2+ \int_{\partial^0 {\widetilde{{\mathcal
O}}_{R,L}}}G(v) \qquad \text{for every}\:\: v\in
\widetilde{H}^1_{0}(\widetilde{{\mathcal O}}_{R,L}).
$$

 Next, we prove the existence of a minimizer of this
functional among functions in $ \widetilde{H}^1_{0}(\widetilde{{\mathcal O}}_{R,L})$. Recall that we assume
condition \eqref{h2} on $G$, that is,
$$
G(\pm 1)=0 \;\; \text{ and } G>0 \text{ in } (-1,1).
$$

We define a continuous extension $\widetilde{G}$ of $G$ in $\re$
such that
\begin{itemize}
\item $\widetilde{G}=G$ in $[-1,1]$, \item $\widetilde{G}>0$ in
$\re\setminus [-1,1]$, \item $\widetilde{G}$ is even,
\item
$\widetilde{G}$ has linear growth at infinity.
\end{itemize}
We consider the new energy functional
$$\widetilde{{\mathcal E}}_{\widetilde{{\mathcal
O}}_{R,L}}(v)=\int_{\widetilde{{\mathcal O}}_{R,L}}
\frac{1}{2}|\nabla v|^2+ \int_{\partial^0 {\widetilde{{\mathcal
O}}_{R,L}}}\widetilde{G}(v) \qquad \text{for every }\:\: v\in
\widetilde{H}_{0}^1(\widetilde{{\mathcal O}}_{R,L}).
$$
Note that every minimizer $w$ of $\widetilde{{\mathcal
E}}_{\widetilde{\mathcal O}_{R,L}}(\cdot) $ in
$\widetilde{H}_{0}^1(\widetilde{{\mathcal O}}_{R,L})$ such that
$-1\leq w \leq 1$ is also a minimizer of $\mathcal
E_{\widetilde{{\mathcal O}}_{R,L}}(\cdot)$ in the set
$$\{v\in \widetilde H^1_{0}(\widetilde{{\mathcal O}}_{R,L}):
-1\leq v\leq 1\}$$

 We show that $\widetilde{{\mathcal
E}}_{\widetilde{\mathcal O}_{R,L}}(\cdot) $ admits a minimizer in
$\widetilde{H}_{0}^1(\widetilde{{\mathcal O}}_{R,L})$. Indeed,
by the properties of
$\widetilde{G}$, it follows that $\widetilde{{\mathcal
E}}_{\widetilde{\mathcal O}_{R,L}}(\cdot) $ is well-defined,
bounded below and coercive in
$\widetilde{H}_{0}^1(\widetilde{{\mathcal O}}_{R,L})$. Hence, using
the compactness of the inclusion
$\widetilde{H}_{0}^1(\widetilde{{\mathcal O}}_{R,L})\subset\subset
\widetilde L^2(\partial^0 \widetilde{{\mathcal O}}_{R,L}) $, taking a
minimizing sequence $\{v_{R,L}^k\}\in
\widetilde{H}_{0}^1(\widetilde{{\mathcal O}}_{R,L})$ and a
subsequence convergent in $\widetilde L^2(\partial^0 \widetilde{{\mathcal
O}}_{R,L}) $, we conclude that $\widetilde{{\mathcal
E}}_{\widetilde{{\mathcal O}}_{R,L}}(\cdot) $ admits an absolute
minimizer $v_{R,L}$ in $\widetilde{H}_{0}^1(\widetilde{{\mathcal
O}}_{R,L})$.

Note moreover that, without loss of generality,  we may assume
that $0\leq v_{R,L}^k\leq 1$ in $\widetilde{{\mathcal O}}_{R,L}$
because, if not, we can replace the minimizing sequence
$v_{R,L}^k$ with the sequence $\min\{|v_{R,L}^k|,1\}\in
\widetilde{H}_{0}^1(\widetilde{{\mathcal O}}_{R,L})$. Indeed, it is
also minimizing because $\widetilde{G}$ is even and $\widetilde{G}\geq
\widetilde{G}(1)$. Then the absolute minimizer $v_{R,L}$ is such that
$0\leq v_{R,L}\leq 1$ in $\widetilde{{\mathcal O}}_{R,L}$.

Next, we can consider perturbations $v_{R,L}+\xi$ of $v_{R,L}$,
with $\xi$ depending only on $s,\:t$ and $\lambda$, and having
compact support in $\widetilde{{\mathcal O}}_{R,L}\cap \{t>0\}$.
In particular $\xi$ vanishes in a neighborhood of $\{t=0\}$. Since
the problem $(\ref{eq2-saddle})$ in $(s,t,\lambda)$ coordinates is the first
variation of ${\mathcal E}_{\widetilde{{\mathcal O}}_{R,L}}(v) $
---recall that ${\mathcal E}$ has the form \eqref{enerst} on
$\widetilde{H}_{0}^1$ functions--- and the equation is not singular
away from $\{s=0\}$ and $\{t=0\}$, we deduce that $v_{R,L}$ is a
solution of $(\ref{eqst})$ in $\widetilde{{\mathcal O}}_{R,L}\cap
\{t>0\}$.

We now prove that $v_{R,L}$ is also a solution in all of
$\widetilde{{\mathcal O}}_{R,L}$, that is, also across $\{t=0\}$.
To see this for dimensions $2m+1\geq 5$, let $\xi_{\varepsilon}$
be a smooth function of $t$ alone being identically $0$ in
$\{t<\varepsilon/2\}$ and identically $1$ in $\{t>\varepsilon\}$.
Let $\varphi\in C^{\infty}_0(\widetilde{{\mathcal
O}}_{R,L}\cup\partial^0\widetilde{{\mathcal
O}}_{R,L} )$, we multiply the equation $-\Delta
v_{R,L}=0$ by $\varphi\xi_{\varepsilon}$ and integrate by parts to
obtain

$$\int_{\widetilde{{\mathcal O}}_{R,L}}\nabla v_{R,L}\nabla
\varphi\:\xi_{\varepsilon}+\int_{\widetilde{{\mathcal
O}}_{R,L}\cap \{t<\varepsilon\}}\nabla v_{R,L}\:\varphi\:\nabla
\xi_{\varepsilon}+\int_{\partial^0\widetilde{{\mathcal
O}}_{R,L}}\partial_{\lambda}v_{R,L}\:\varphi\:\xi_{\varepsilon}=0.$$
Reminding that $v_{R,L}$ satisfies the Neumann condition
$-\partial_{\lambda}v_{R,L}=f(v_{R,L})$ on
$\partial^0\widetilde{{\mathcal O}}_{R,L}$, we get
\begin{equation}\label{cut}\int_{\widetilde{{\mathcal O}}_{R,L}}\nabla v_{R,L}\nabla
\varphi\:\xi_{\varepsilon}+\int_{\widetilde{{\mathcal
O}}_{R,L}\cap \{t<\varepsilon\}}\nabla v_{R,L}\:\varphi\:\nabla
\xi_{\varepsilon}=\int_{\partial^0\widetilde{{\mathcal
O}}_{R,L}}f(v_{R,L})\:\varphi\:\xi_{\varepsilon}.\end{equation}

We conclude by seeing that the second integral on the left hand
side goes to zero as $\varepsilon\to 0$. Indeed, by
Cauchy-Schwartz inequality,
\begin{eqnarray}
&&\left| \int_{{\widetilde{{\mathcal O}}_{R,L}}\cap
\{t<\varepsilon\}}  \nabla v_{R,L}\varphi \nabla\xi_{\varepsilon}
 \;dx d\lambda \right|^2
\nonumber \\
&& \hspace{3em} \leq C \int_{{\widetilde{{\mathcal O}}_{R,L}}\cap \{t<\varepsilon\}}
|\nabla v_{R,L}|^2 \;dx d\lambda \int_{{\widetilde{{\mathcal
O}}_{R,L}}\cap \{t<\varepsilon\}} |\nabla\xi_{\varepsilon}|^2 \;dx
d\lambda.
\end{eqnarray}

Since $|\nabla\xi_{\varepsilon}|^2 \leq C/\varepsilon^2$,
$|{\widetilde{{\mathcal O}}_{R,L}}\cap \{t<\varepsilon\}|\leq C_R
\varepsilon^m\:L$, and $m\geq 2$, the second factor in the
previous bound, is bounded independently of $\varepsilon$. At the
same time, the first factor tends to zero as $\varepsilon\to 0$,
since $|\nabla v_{R,L}|^2$ is integrable in ${\widetilde{{\mathcal
O}}_{R,L}}$.

In dimension $2m+1=3$, the previous proof does not apply and we
argue as follows. We consider perturbations $\xi\in
\tilde{H}_{0}^1(\widetilde{{\mathcal O}}_{R,L})$ which do not
vanish on $\{t=0\}$. Considering the first variation of energy and
integrating by parts, we find that the boundary flux
$s^{m-1}t^{m-1}\partial_tv_{R,L}=\partial_tv_{R,L}$ (here $m-1=0$)
must be identically 0 on $\{t=0\}$. This implies that $v_{R,L}$ is
a solution also across $\{t=0\}$.

We have established the existence of a solution $v_{R,L}$ in
$\widetilde{{\mathcal O}}_{R,L}$ with $0\leq v_{R,L}\leq 1$.
Considering the odd reflection of $v_{R,L}$ with respect to
${\mathcal C}\times \re^+$,
$$
v_{R,L}(s,t,\lambda)=-v_{R,L}(t,s,\lambda),
$$
we obtain a solution in $B_R\setminus \{0\}\times (0,L)$. Using
the same cut-off argument as above, but choosing now
$1-\xi_{\varepsilon}$ to have support in the ball of radius
$\varepsilon$ around $0$, we conclude that $v_{R,L}$ is also
solution around $0$, and hence in all of $B_R\times (0,L)$. Here,
the cut-off argument also applies in dimension $3$.

We now wish to pass to the limit in $R$ and $L$, and obtain a
solution in all of $\re^{2m+1}_+$. Let $S>0$, $L'>0$ and consider the family
$\{v_{R,L}\}$ of solutions in $B_{S+2}\times[0,L'+2]$, with
${R>S+2}$ and $L>L'+2$. Since $|v_{R,L}|\leq 1$, regularity results proved in \cite{C-SM}, applied in
$B_{2}\times[0,2]$ where $B_2$ is centered at points in
$\overline{B}_S\times [0,L']$, give a uniform
$C^{2,\alpha}(\overline{B}_S\times[0,L'])$ bound for $v_{R,L}$
(uniform with respect to $R$ and $L$). We have
\begin{equation}\label{grad1}
|\nabla v_{R,L}|\leq C \quad\text{ in } B_S\times [0,L'],
\qquad\text{for all } R>S+2,\:\:L>L'+2
\end{equation}
for some constant $C$ independent of $S$, $R$, $L$ and $L'$. Moreover since $v_{R,L}$ is harmonic and bounded we have that
\begin{equation}\label{grad2}
 |\nabla v_{R,L}(x,\lambda)|\leq \frac{C}{\lambda}\quad \mbox{in}\:\:B_R\times (1,L).
\end{equation}
Choose now
$L=R^{\gamma}$, with $1/2<\gamma<1$ (this choice will be used later to prove that the solution that we construct is not identically zero). By the
Arzel\`a-Ascoli Theorem, a subsequence of $\{v_{R,R^{\gamma}}\}$ converges in
$C^2(\overline{B}_S\times [0,S^{\gamma}])$ to a solution in
$B_S\times(0,S^{\gamma})$. Taking $S=1,2,3,\ldots$ and making a
Cantor diagonal argument, we obtain a sequence $v_{R_j,R_j^{\gamma}}$
converging in $C^2_{{loc}}(\re^{2m+1}_+)$ to a solution $v\in
C^2(\re^{2m+1}_+)$. By construction we have found a solution $v$
in $\re^{2m+1}_+$ depending only on $s,\:t$ and $\lambda$, such
that $v(s,t,\lambda)=-v(t,s,\lambda)$, $|v|\leq 1$ and $v\geq 0$
in $\{s>t\}$. We want to prove now that $|v|<1$. Indeed,
remind that $v$ satisfies
$$\begin{cases}
\Delta v=0&\mbox{in}\:\: \re^{2m+1}_+\\
-\partial_\lambda v=f(v)&\mbox{on}\:\: \partial{\re^{2m+1}}
\end{cases}$$
Since $f( 1)=0$ and $v$ is not identically $1$ (because $v\equiv 0$ on $\mathcal C \times \re^+$), using that $v\leq 1$ and applying
the maximum principle and Hopf's Lemma, we conclude that $v<1$. In
the same way we prove that $v>-1$.

It only remains to prove that $v\not \equiv 0$ in $\re^{2m+1}_+$. Then, the strong maximum principle and Hopf's Lemma lead to $v>0$ in $\{s>t\}\times \re^+$ since $f(0)=0$ and $v\geq 0$ in $\{s>t\}\times \re^+$.

To prove that $v\not \equiv 0$ in $\re^{2m+1}_+$, we establish an energy estimate for the saddle solution constructed above, which is not sharp, but it is enough to prove $v\not \equiv 0$ in $\mathcal O=\{s>t\}\times \re^+$.

We use a comparison argument, based on the minimality property of $v_{R,L}$ in the set $\widetilde{\mathcal O}_{R,L}$.

Let $1/2 <\gamma < 1$ as above and $\beta$ be a positive real number depending only on $\gamma$ and such that $1/2\leq \beta<\gamma<1$. Let $S<R-2$, then $S^{\gamma}<L$
since we have chosen $L=R^\gamma$. We consider a $C^1$ function
$g:\widetilde {\mathcal O}_{S,S^{\gamma}}\rightarrow \re$ defined as follows:
$$g(x,\lambda)=g(s,t,\lambda)=\eta(s,t)\min\left\{1,\frac{s-t}{\sqrt 2}\right\}+(1-\eta(s,t))v_{R,L}(s,t,\lambda),$$
where $\eta$ is a smooth function depending only on $r^2=s^2+t^2$ such that $\eta\equiv 1$ in $B_{S-1}$ and $\eta\equiv 0$ outside $B_S$. Observe that $g$ agrees with $v_{R,L}$ on the lateral boundary of $\widetilde {\mathcal O}_{S,S^{\gamma}}$ and $g$ is identically $1$ inside
$(\mathcal O_{S-1}\cap \{(s-t)/\sqrt 2>1\})\times (0,S^{\gamma})$. By \eqref{grad1} and \eqref{grad2}, we have that
\begin{equation}\label{grad_g}
|\nabla g(x,\lambda)|\leq \frac{C}{\lambda+1}\quad \mbox{for every} \:\:(x,\lambda) \in \widetilde {\mathcal O}_{S,S^{\gamma}}.
\end{equation}
Next we consider a $C^1$ function  $\xi:(0,S^\gamma)\rightarrow (0,+\infty)$, such that
$$\xi(\lambda)=\begin{cases} 1 & \mbox{if}\:\:0<\lambda\leq S^{\gamma}-S^{\beta}\\
\displaystyle \frac{\log {S^\gamma}-\log \lambda}{\log {S^\gamma}-\log {(S^\gamma-S^\beta)}} & \mbox{if}\:\:S^{\gamma}-S^{\beta}<\lambda\leq S^{\gamma}
                \end{cases}
$$

Then, we define $w:\widetilde {\mathcal O}_{S,S^{\gamma}}\rightarrow (-1,1)$ as follows
\begin{equation}\label{defw}
 w(x,\lambda)=\xi(\lambda)g(x,\lambda)+[1-\xi(\lambda)]v_{R,L}(x,\lambda).
\end{equation}

Observe that $w$ agree with $v_{R,L}$ on $\partial^+\widetilde {\mathcal O}_{S,S^{\gamma}}$ and $w\equiv 1$ in $\widetilde {\mathcal O}_{S-1,S^{\gamma}-S^\beta}$.
We extend $w$ to be identically equal to $v_{R,L}$ in $\widetilde {\mathcal O}_{R,L}\setminus \widetilde{ \mathcal O}_{S,S^{\gamma}}$.
By minimality of $v_{R,L}$ in $\widetilde{ \mathcal O}_{R,L}$, we have
$$\mathcal E_{ \widetilde {\mathcal O}_{R,L}}(v_{R,L})\leq \mathcal E_{ \widetilde {\mathcal O}_{R,L}}(w).$$
Thus, since $w=v_{R,L}$ in $\widetilde {\mathcal O}_{R,L}\setminus \widetilde {\mathcal O}_{S,S^{\gamma}}$, we get
$$\mathcal E_{ \widetilde {\mathcal O}_{S,S^{\gamma}}}(v_{R,L})\leq \mathcal E_{ \widetilde {\mathcal O}_{S,S^{\gamma}}}(w).$$
We give now an estimate for $\mathcal E_{ \widetilde {\mathcal O}_{S,S^{\gamma}}}(w).$
First, observe that, since $w\equiv 1$ on $\mathcal O_{S-1}$, then
\begin{equation}\label{pot}
\int_{\mathcal O_S} G(w)= \int_{\mathcal O_S \setminus \mathcal O_{S-1}}G(w)\leq C|\mathcal O_S \setminus \mathcal O_{S-1}|\leq CS^{2m-1}.
\end{equation}

Next, we give a bound for the Dirichlet energy of $w$. We have
\begin{eqnarray}\label{dir}\int_{\widetilde {\mathcal O}_{S,S^{\gamma}}}|\nabla w(x,\lambda)|^2 dx d\lambda&=& \int_{\widetilde {\mathcal O}_{S,S^{\gamma}-S^\beta}}|\nabla w(x,\lambda)|^2 dx d\lambda \nonumber\\
&&\hspace{1em}+\int_{\widetilde {\mathcal O}_{S,S^{\gamma}}\setminus \widetilde {\mathcal O}_{S,S^{\gamma}-S^\beta}}|\nabla w(x,\lambda)|^2 dx d\lambda.\end{eqnarray}

Since $w\equiv 1$ in $\widetilde {\mathcal O}_{S-1,S^{\gamma}-S^\beta}$, we get
\begin{equation}\label{dir1}
\int_{\widetilde {\mathcal O}_{S,S^{\gamma}}}|\nabla w(x,\lambda)|^2 dx d\lambda \leq CS^{2m-1+\gamma} +\int_{\widetilde {\mathcal O}_{S,S^{\gamma}}\setminus \widetilde {\mathcal O}_{S,S^{\gamma}-S^\beta}}|\nabla w(x,\lambda)|^2 dx d\lambda.
\end{equation}

Consider now the integral on the right-hand side of \eqref{dir1}. By the definition \eqref{defw} of $w$, we have that
$$ |\nabla w(x,\lambda)|^2\leq |\xi'(\lambda)|^2[g(x,\lambda)+v_{R,L}(x,\lambda)]^2+\{|\nabla g|^2+|\nabla v_{R,L}(x,\lambda)|^2\}[1+\xi(\lambda)]^2.$$

Integrating in $\widetilde {\mathcal O}_{S,S^{\gamma}}\setminus \widetilde {\mathcal O}_{S,S^{\gamma}-S^\beta}$, using that $g,\:|\nabla g|,\:v,$ and $\xi$ are bounded, the definition of $\xi$, and the gradient bounds \eqref{grad2} and \eqref{grad_g} for $v_{R,L}$ and for $g$, we get
\begin{eqnarray}\label{dir2}
\int_{\widetilde {\mathcal O}_{S,S^{\gamma}}\setminus \widetilde {\mathcal O}_{S,S^{\gamma}-S^\beta}} |\nabla w(x,\lambda)|^2 &\leq& C\int_{\mathcal O_S}\int_{S^{\gamma}-S^\beta}^{S^{\gamma}}|\xi'(\lambda)|^2d\lambda dx + C\int_{\mathcal O_S}\int_{S^{\gamma}-S^\beta}^{S^{\gamma}}\frac{1}{\lambda^2}d\lambda dx\nonumber \\
\displaystyle &\leq& C\left[\frac{1}{\left(\log{\frac{S^\gamma}{S^\gamma-S^\beta}}\right)^2}+1\right]\int_{\mathcal O_S}\int_{S^{\gamma}-S^\beta}^{S^{\gamma}}\frac{1}{\lambda^2}d\lambda dx \nonumber \\
&\leq& CS^{2m}\left[\frac{1}{\left(-\log{(1-S^{\beta-\gamma})}\right)^2}+1\right]\left[\frac{1}{S^{\gamma}-S^\beta}-\frac{1}{S^{\gamma}}\right]\nonumber  \\
&\leq& CS^{2m}\cdot S^{2(\gamma-\beta)}\cdot S^{-\gamma} \leq CS^{2m+\gamma-2\beta},
\end{eqnarray} where $C$ denotes different positive constants independent on $S$.

Combining \eqref{pot}, \eqref{dir1} and \eqref{dir2}, we get
\begin{equation}\label{total}\mathcal E_{\widetilde {\mathcal O}_{S,S^{\gamma}}}(w)\leq C(S^{2m-1}+S^{2m-1+\gamma}+S^{2m+\gamma-2\beta}).\end{equation}

Since, by hypothesis, $\gamma$ and $\beta=\beta (\gamma)$ satisfy $1/2\leq \beta<\gamma<1$, then there exists $\varepsilon=\varepsilon(\gamma)>0$ such that
$$\mathcal E_{\widetilde {\mathcal O}_{S,S^{\gamma}}}(w)\leq CS^{2m-\varepsilon}.$$
Thus by minimality of $v_{R,L}$, we get
$$\mathcal E_{\widetilde {\mathcal O}_{S,S^{\gamma}}}(v_{R,L})\leq CS^{2m-\varepsilon}.$$

We now let $R$ and $L=R^{\gamma}$ tend to infinity to obtain
$$\mathcal E_{\widetilde {\mathcal O}_{S,S^{\gamma}}}(v)\leq CS^{2m-\varepsilon}.$$
Note that this bound, after odd reflection with respect to $\mathcal C$, leads to the energy bound \eqref{en-saddle}
$$\mathcal E_{C_{S,S^{\gamma}}}(v)\leq CS^{2m-\varepsilon}.$$

Using this estimate we prove the claim. Suppose that $v \equiv 0$. Then we would have
$$c_m G(0)S^{2m}=\mathcal E_{C_{S,S^{\gamma}}}(v)\leq CS^{2m-\varepsilon}.$$
This is a contradiction for $S$ large, and thus $v\not \equiv 0$.

We give now the proof of the last part of the statement, that is, we prove stability of saddle-shaped solutions under perturbations vanishing on $\mathcal C\times (0,+\infty)$.

Since $f(0)=0$, concavity leads to $f'(w)\leq f(w)/w$ for all real
numbers $w\in (0,1)$. Hence we have
$$\begin{cases}
-\Delta v=0 & \mbox{in}\:\: \mathcal{\widetilde{O}}\\
\displaystyle -\partial_\lambda v\geq f'(v)v& \mbox{on}\:\:
\mathcal{O}\times \{0\}. \end{cases}$$

By a simple argument (see the proof of Proposition 4.2 of
\cite{AAC}), it follows that the value of the quadratic form
$Q_v(\xi)$ is nonnegative for all $\xi\in C^1$ with compact
support in $\mathcal {\widetilde{O}}\cup
\partial^0\mathcal{\widetilde{O}}$ (and not necessarily depending only on
$s$, $t$ and $\lambda$). Indeed, multiply the equation $-\Delta
v=0$ by $\xi^2/v$, where $\xi\in C^1(\re^{2m+1}_+)$ with compact
support in $\mathcal {\widetilde{O}}\cup
\partial^0\mathcal{\widetilde{O}} $, and
integrate by parts in $\mathcal{\widetilde{O}}$, we get:
\begin{eqnarray} 0&=&\int_{\mathcal{\widetilde O}} (-\Delta
v)\frac{\xi^2}{v}=\int_{\mathcal {\widetilde O}}\nabla v\cdot
\nabla \xi \frac{2\xi}{v}\nonumber
\\ &&\quad -\int_{\mathcal {\widetilde O}} |\nabla
v|^2\frac{\xi^2}{v^2}+\int_{\partial^0\mathcal
O}\frac{\xi^2}{v}\frac{\partial v}{\partial \lambda} \nonumber \\
&\leq& \int_{\mathcal {\widetilde O}} |\nabla \xi|^2-
\int_{\partial^0\mathcal O}f'(v)\xi^2 =Q_v(\xi).\nonumber \end{eqnarray}

By an approximation argument, the same holds for all $\xi\in C^1$
with compact support in the closure of ${\widetilde{\mathcal O}}$ and
vanishing on  ${\mathcal C}\times \re^+$. Finally, by odd symmetry
with respect to ${\mathcal C}\times \re^+$, the same is true for
all $C^1$ functions $\xi$ with compact support in
$\overline{\re_+^{2m+1}} $ and vanishing on~${\mathcal C}\times
\re^+$.
\end{proof}
\begin{oss}
Observe that, if $\gamma \rightarrow 1$, estimate \eqref{total} tends to
$$\mathcal E_{C_{S,S}}(v)\leq CS^{2m}.$$
This is a not sharp energy estimate, indeed in Theorem 1.7 of \cite{C-Cinti1}, Cabr\'e and the author prove that saddle solutions $v$ satisfy
$$\mathcal E_{C_{S,S}}(v)\leq CS^{2m-1}\log S.$$
\end{oss}

\section{Supersolution and subsolution for $A_{1/2}$}

In \cite{Tan}, Cabr\'e and Tan introduced the operator $A_{1/2}$, which is
the square root of the Laplacian for functions defined on a bounded set and that
vanish on the boundary.
Let $u$ be defined in a bounded set $H\subset \re^n$ and $u\equiv 0$ on $\partial
H$. Consider the harmonic extension $v$ of $u$ in the half-cylinder
$H\times(0,\infty)$ vanishing on the lateral boundary
$\partial H\times[0,\infty)$. Define the operator
$A_{1/2}$ as follows
\begin{equation}\label{def-A} A_{1/2}u:=-\partial_\lambda v_{|H \times \{0\}}.\end{equation}

Then, since $\partial_{\lambda}v$ is harmonic and also vanishes on
the lateral boundary,
as for the case of the all space, the Dirichlet-Neumann map of the harmonic extension $v$ on the bottom of the half cylinder
 is the square root of the Laplacian.  That is, we have the property:
\[
A_{1/2}\circ A_{1/2}=-\Delta_{H}\
\]
where $-\Delta_{H}$ is the Laplacian in $H$ with zero
Dirichlet boundary value on $\partial H$.

 Hence, we can
study the problem \begin{equation}\label{pb-A}\begin{cases}
A_{1/2}u=f(u)&
\mbox{in}\:\:H\\
u=0 & \mbox{on}\:\: \partial H\\
u>0 &\mbox{in}\:\:H,\end{cases}\end{equation} by
studying the local problem
\begin{equation}\label{pb-A2}
\begin{cases}
 -\Delta v=0 &  \mbox{in}\quad  \Omega=H\times (0,\infty)\\
 v=0 & \mbox{on}\quad \partial_{L}\Omega=\partial H\times [0,\infty)\\
 v>0 &  \mbox{in}\quad \Omega\\
 -\partial_\lambda v=f(v) & \mbox{on} \quad H\times \{0\}.

\end{cases}
\end{equation}

In \cite{Tan} some results (Lemma 3.2.3 and Lemma 3.2.4) need to assume that
$H$ is bounded. But for our aim, definition \eqref{def-A} is
enough and it can be given also in the case that $H$ is not bounded.
Thus, we can consider problem \eqref{pb-A} and \eqref{pb-A2} for a
general open set $H\subset \re^n$.

In this section we give a subsolution and supersolution for the problem
\begin{equation}\begin{cases}
A_{1/2}u=f(u)&
\mbox{in}\:\:\mathcal{O}\\
u=0 &\mbox{on}\:\: \partial \mathcal{O}\\
u>0 & \mbox{in}\:\:\mathcal{O}.\end{cases}\end{equation}

In what follows it will be useful to use the new
variables:
\begin{equation}
\begin{cases}
\displaystyle y  =   \frac{s+t}{\sqrt{2}} \\
\displaystyle z  =  \frac{s-t}{\sqrt{2}}
\end{cases}.
\end{equation}
Note that $|z|\leq y$ and that we may write the Simons cone as
${\mathcal C}=\{z=0\}$.

If we take into account these new variables, problem
($\ref{eqst}$) becomes
\begin{equation}\label{eqyz}\begin{cases}
\displaystyle v_{yy}+v_{zz}+v_{\lambda\lambda}+\frac{2(m-1)}{y^2-z^2}\left(yv_y-zv_z\right)=0& \mbox{in}\:\:\re^{2m+1}_+\\
-\partial_\lambda v=f(v)& \mbox{on}\:\:\partial{\re^{2m+1}_+}
\end{cases}
\end{equation}

We give the definition of supersolution and subsolution for 
problem (\ref{pb-A}) by using the associated local formulation
(\ref{pb-A2}).

\begin{defin}
a) We say that a function $w$, defined on $H\times [0,+\infty)$, $w\equiv 0$ on $\partial H\times [0,+\infty)$ is a
\emph{supersolution} (\emph{subsolution}) for problem
(\ref{pb-A2}) if
$$\begin{cases}
-\Delta w \geq\:(\leq)\:0& \mbox{in}\:\;H \times (0,+\infty)\\
w>0 & \mbox{in}\:\;H \times (0,+\infty)\\
\displaystyle -\partial_\lambda w \geq\:(\leq)\:f(w) &
\mbox{on}\:\;H \times\{0\}. \end{cases}
$$

b) We say that a function $u$, defined on $H$, $u\equiv 0$ on
$\partial H$, $u>0$ in $H$, is a \emph{supersolution}
(\emph{subsolution}) for problem (\ref{pb-A}) if its harmonic
extension $v$ such that $v\equiv 0$ on $\partial H \times
[0,+\infty)$, is a supersolution (subsolution) for problem \eqref{pb-A2}.
\end{defin}

\begin{lem}
The following assertions are equivalent:
\begin{itemize}
\item[i)] $u$ is a subsolution (supersolution) for problem
(\ref{pb-A}); \item[ii)] there exists an extension $w$ of $u$ on
$H \times (0,+\infty)$ vanishing on $\partial H\times (0,+\infty)$, such that $w$ is a subsolution
(supersolution) for problem (\ref{pb-A2}).\end{itemize}
\end{lem}

\begin{proof}
The first implication $i)\;\Rightarrow\: ii)$ is trivial.

It remains to show that $ii)\; \Rightarrow \:i)$. We consider the
case of supersolution (the argument for subsolution is analog). Suppose that there exists a function $w$
defined on $\re^{n+1}_+$ such that:
$$\begin{cases}
-\Delta w \geq 0& \mbox{in}\:\;H \times (0,+\infty)\\
w\equiv 0& \mbox{on}\;\:\partial H \times (0,+\infty)\\
w>0 & \mbox{in}\:\;H \times (0,+\infty)\\
w(x,0)=u(x)& \mbox{on}\:\;H \times\{0\}\\
 -\partial_\lambda w \geq f(w) &
\mbox{on}\:\;H \times\{0\}. \end{cases}
$$
Now consider the harmonic extension $v$ of $u$ in $H \times
(0,+\infty)$, with $v\equiv 0$ on $\partial H \times
(0,+\infty)$. Then by the maximum principle we have that $v\leq w$
in $H \times (0,+\infty)$. This implies that
$$-\partial_\lambda v\geq -\partial_\lambda w\;\;\mbox{on}\;\;\partial H \times (0,+\infty)$$ and hence that
$$-\partial_\lambda v\geq f(v) \;\;\mbox{on}\;\;\partial H \times (0,+\infty).$$
\end{proof}

We recall that in \cite{C-SM} it is
proven that, under hypothesis \eqref{h2}, there exists a layer solution (i.e., a monotone increasing solution, from $-1$ to $1$), for 
problem (\ref{eq2-saddle}) in dimension $n=1$. Normalizing it to vanish at $\{x=0\}$, we call it $u_0$ (see \eqref{eq-layer}).

Moreover we remind that $|s-t|/\sqrt 2$ is the distance to the Simons cone (see \cite{CT}).


We can give now the following proposition. The first part of the statement, which gives a supersolution for problem \eqref{pb-A} in $H=\mathcal O$, is equivalent to Proposition \ref{super-ext} in the Introduction.
\begin{prop} \label{super-sub}
Let $f$ satisfy hypothesis \eqref{h1}, \eqref{h2}, \eqref{h3}. Let $u_0$ be the layer solution, vanishing at the origin, of problem \eqref{eq1-saddle} in $\re$.

Then, the function $\displaystyle u_0(z)=u_0(s-t)/\sqrt 2$ is a
supersolution of problem (\ref{pb-A}) in the set
$H=\mathcal O=\{s>t\}$.

\end{prop}

\begin{oss}\label{f(u)/u}
We observe that, if $f$ satisfies hypothesis \eqref{h1}, \eqref{h2}, \eqref{h3}, then $f(\rho)/\rho$ is non-increasing in $(0,1)$.
Indeed, given $0<\rho<1$, there exists $\rho_1$, with $0<\rho_1<\rho$, such that
$$\frac{f(\rho)}{\rho}=\frac{f(\rho)-f(0)}{\rho-0}=f'(\rho_1)>f'(\rho).$$
Therefore
$$\left(\frac{f(\rho)}{\rho}\right)'=\frac{f'(\rho)\rho-f(\rho)}{\rho^2}=\frac{f'(\rho)-f'(\rho_1)}{\rho}<0.$$
\end{oss}

\begin{proof}[Proof of Proposition \ref{super-sub}]
We begin by considering the function $\displaystyle v_0((s-t)/\sqrt 2,\lambda)$  and
we show that it is a supersolution of the problem (\ref{pb-A2}) in
the set $\widetilde{\mathcal{O}}$.

First, we remind that the
problem (\ref{pb-A2}) in the $(s,t,\lambda)$ variables reads

\begin{equation}\begin{cases}
\displaystyle -(v_{ss}+v_{tt}+v_{\lambda\lambda})-(m-1)\left(\frac{v_s}{s}+\frac{v_t}{t}\right)=0 &\mbox{in}\;\;
\widetilde{\mathcal{O}} \\
v=0 &\mbox{on}\;\;\mathcal C\times [0,+\infty)\\
-\partial_\lambda v=f(v)& \mbox{on}\;\;
\widetilde{\mathcal{O}}\cap \{\lambda=0\}\\
v>0& \mbox{in}\;\;\widetilde{\mathcal{O}}. \end{cases}
\end{equation}
By a direct computation, we have that $\displaystyle v_0((s-t)/\sqrt 2,\lambda)$ is superharmonic in the set $\{(s,t,\lambda):s>t>0\}$
and satisfies the Neumann condition $-\partial_{\lambda}v=f(v)$. In dimension $2m+1 \geq 5$
there is nothing else to be checked, by a cut-off argument used as
in (\ref{cut}).

In dimension $2m+1=3$, $\displaystyle v_0((s-t)/\sqrt 2,\lambda)$ is a
supersolution in $\widetilde{\mathcal{O}}$ because the outer flux
$\displaystyle -\partial_t v_0((s-t)/\sqrt 2,\lambda)=\partial_x
v_0\left((s-t)/\sqrt 2,\lambda\right)>0$ is positive.\end{proof}
\begin{oss}
Observe that in dimension $2m+1=3$, $\displaystyle v_0((s-t)/\sqrt 2,\lambda)$
is a solution of problem (\ref{eq2-saddle}) away from the sets
$\{s=0\}$, $\{t=0\}$, while in higher dimensions it is a strict
supersolution.
\end{oss}

\begin{cor}\label{super-k}
Let $f$ satisfy hypothesis \eqref{h1}, \eqref{h2}, \eqref{h3}. Let $u_0$ be the layer solution, vanishing at the origin, of problem \eqref{eq1-saddle} in $\re$ and suppose $K\geq 1$.

Then, the function $\min\{\displaystyle Ku_0(z),1\}=\min\{Ku_0(s-t/\sqrt 2),1\}$ is a
supersolution of problem (\ref{pb-A}) in the set
$\mathcal O=\{s>t\}$.
\end{cor}
\begin{proof}
Proceeding as in the proof of Proposition \ref{super-sub}, we consider the function \\ $\min\{Kv_0(z,\lambda),1\}$. To prove that it is a supersolution of problem \eqref{pb-A2} in $\widetilde {\mathcal O}$, it is enough to prove that it is a supersolution  of problem \eqref{pb-A2} in the set $\{(x,\lambda)\in \widetilde {\mathcal O}:Kv_0(z,\lambda)<1\}$.

First of all, in the proof of Proposition \ref{super-sub}, we have seen that $v_0(z,\lambda)$ is superharmonic in $\widetilde {\mathcal O}$, and thus $\min\{Kv_0(z,\lambda),1\}=Kv_0(z,\lambda)$ is superharmonic in the set $\{(x,\lambda)\in \widetilde {\mathcal O}:Kv_0(z,\lambda)<1\}$.

Moreover
$$-\partial_{\lambda}(Kv_0(z,0))=Kf(v_0(z,0))\quad \mbox{on}\:\:\{(x,0)\in \widetilde {\mathcal O}:Kv_0(z,0)<1\}.$$
By Remark \ref{f(u)/u}, we have that $f(u)/u$ is decreasing and then for every $K\geq 1$ we get
$$\frac{Kf(u_0)}{Ku_0}=\frac{f(u_0)}{u_0}\geq \frac{f(Ku_0)}{Ku_0}\quad \mbox{if}\:\:Ku_0<1.$$
This let us to conclude the proof, indeed
\begin{equation*}
-\partial_{\lambda}(Kv_0(z,0))=Kf(v_0(z,0))\geq f(Kv_0(z,0))\quad \mbox{on}\:\:\{(x,0)\in \widetilde {\mathcal O}:Kv_0(z,0)<1\}.
\end{equation*}

\end{proof}

\section{The operator $D_{H,\varphi}$ and maximum principles}

In what follows we need to introduce a new nonlocal operator $D_{H,\varphi}$,
which is the analogue of $A_{1/2}$ but it can be applied to
functions which do not vanish on the boundary of $H$.

Suppose that $u$ and $\varphi$ are functions defined in $\overline H \subset \re^n$, such that $u=\varphi$
on $\partial H$. As in the case of $A_{1/2}$ we want to
consider the harmonic extension $v$ of $u$ in the cylinder
$\Omega=H \times (0,+\infty)$ and we have to give Dirichlet data
on the lateral boundary of the cylinder $\partial_L \Omega=\partial
H \times (0,+\infty)$. We do it in the following way: we put
$v(x,\lambda)=\varphi(x)$ for every $(x,\lambda)\in \partial_L \Omega$.

As before we define $D_{H,\varphi}$ as follows:
$$D_{H,\varphi}u:=-\partial_{\lambda}v_{|\Omega\times\{0\}}.$$

We observe that, since $v$ is independent on
$\lambda$ on $\partial_L \Omega$, we have $v_{\lambda}=0$ on the lateral
boundary. Thus, we can apply the operator
$A_{1/2}$ to $v_\lambda(x,0)$ and we get, as before
$$A_{1/2}\circ D_{H,\varphi}=-\Delta_{H,\varphi}$$
where $-\Delta_{H,\varphi}$ is the Laplacian in $H$ with
Dirichlet boundary value $\varphi$.

If we have a nonlocal problem of the type
$$\begin{cases} D_{H,\varphi}u=f(u)& \mbox{in}\:\:H\\
u=\varphi& \mbox{on} \:\:\partial H,\end{cases}$$ then it
can be restated in the local problem
\begin{equation}\label{ext-D}\begin{cases}
-\Delta v=0 & \mbox{in}\:\:\Omega\\
v(x,\lambda)=\varphi(x)& \mbox{on}\:\:\partial_L \Omega\\
\displaystyle-\partial_\lambda v=f(v)& \mbox{on}\:\:H
\times \{0\}.\end{cases}\end{equation}

Observe that the operator
$D_{H,\varphi}$ coincides with $A_{1/2}$  if the boundary data $\varphi$
is identically zero.

Next, we give some maximum principles for the operator $D_{H,\varphi}$.
\begin{lem}\label{MP1}
Let $\Omega=H\times \re^+$ be a cylinder in $\re^{n+1}_+$, where $H\subset \re^n$ is a bounded domain. Let $v \in C^2(\Omega)\cap C^(\overline \Omega)$ be a bounded harmonic function in $\Omega.$
Then,
$$\inf_{\Omega}v=\inf_{\partial \Omega} v.$$
\end{lem}
\begin{proof}
Substracting a constant from $v$, we may assume that $v$ is nonnegative on $\partial \Omega$ and we need to show $v\geq 0$ in $\Omega$.

We follow a classical argument based on the construction of a strictly positive harmonic function $\psi$ in $\Omega$ tending to infinity as $|(x,\lambda)|\rightarrow \infty$. We proceed in the following way.

First, since $H\subset \re^n$ is bounded, there exists a ball $B_R$ of radius $R$ in $\re^n$ such that $\overline H\subset B_R$. Let $\mu_R$ and $\phi_R$ be, respectively, the first eigenvalue and the corresponding eigenfunction of the Laplacian $-\Delta$ in $B_R$ with $0-$Dirichlet value on $\partial B_R$.

We define the function $\psi: B_R\times \re^+ \rightarrow \re$ as follows
$$\psi(x,\lambda)=\phi_R(x) e^{\sqrt {\mu_R} \lambda}.$$

Then the restriction of $\psi$ in $\Omega$ is a strictly positive harmonic function.

Moreover, since $\phi_R$ is bounded, we have that
\begin{equation}\label{lim-psi}
\lim_{|(x,\lambda)|\rightarrow +\infty} \psi(x,\lambda)=\lim_{\lambda\rightarrow +\infty}\psi(x,\lambda)=+\infty.
\end{equation}

We consider now the function $w=v/\psi$.
Then $w$ satisfies
\begin{equation*}
\begin{cases}
 \displaystyle -\Delta w - 2 \frac{\nabla \psi}{\psi}\cdot \nabla w=0 &\mbox{in}\:\:\Omega \\
w\geq 0 & \mbox{on} \:\:\partial \Omega.
\end{cases}
\end{equation*}
Note that $w$ has the same sign as $v$. In addition, by \eqref{lim-psi}, $w(x,\lambda)\rightarrow 0$ as $|(x,\lambda)|\rightarrow +\infty$ and thus, by the strong maximum principle (applied, by a contradiction argument, to a possible negative minimum) $w\geq 0$ in $\Omega$, which implies $v\geq 0$ in $\Omega$.
\end{proof}

From the previous result we deduce the following lemma.
\begin{lem}\label{MP2}
Assume that $u\in C^{2}({H})\cap C(\overline{H})$
satisfies
\begin{equation*}
\left\{
\begin{array}{ll}
D_{H,\varphi}u +c(x)u \ge 0 &\mbox{in}\; H,\\
u=\varphi &\mbox{on}\; \partial H,
\end{array}
\right.
\end{equation*}
where $H$ is a bounded domain in $\re^{n}$ and
$c(x)\ge 0$ in $H$. Suppose that $\varphi\geq 0$ on $\partial
H$. Then $u\ge 0$ in $H$.
\end{lem}
\begin{proof} Consider the harmonic extension $v$ of $u$ in $\Omega=H\times (0,+\infty)$
with Dirichlet data $v(x,\lambda)=\varphi(x)$ on the lateral
boundary $\partial_L \Omega=\partial H \times (0,+\infty)$ (as in
the definition of the operator $D_{H,\varphi}$). We prove that $v\ge
0$ in $\Omega $, then in particular $u\ge 0$ in $H$.

Suppose by contradiction that $v$ is negative somewhere in $\Omega
\times \re^+$. Since $v$ is harmonic, by Lemma \ref{MP1} the
$\inf_{\Omega
} v <0$ will be achieved at some point
$(x_{0},0)\in H\times\{0\}$. Thus, we have
\[
\inf_{\Omega}v=v(x_{0},0)<0.
\]
By Hopf's lemma,
\[
v_{\lambda}(x_{0},0) >0.
\]
It follows $$-v_{\lambda}(x_{0},0)=D_{H,\varphi}v(x_{0},0) <0.$$

Therefore, since $c\ge 0$,
\[
D_{H,\varphi}v(x_{0},0)+c(x_{0})v(x_{0},0)<0.
\]
 This is a contradiction  with the hypothesis $D_{H,\varphi} u+ c(x)u\ge 0$.
\end{proof}

The following corollary follows directly by the previous lemma.

\begin{cor}\label{MP3}
Let $H$ be a bounded domain in $\re^n$.
 Suppose that $u_1$ and $u_2$ are two bounded functions, $u_1,\:u_2\in C^2(H)\cap C(\overline H)$, which satisfy
\begin{equation*}
 \begin{cases}
  D_{H,\varphi}u_1\leq D_{H,\varphi} u_2 & \mbox{in}\:\:H \\
u_1=u_2=\varphi & \mbox{on}\:\:\partial H .
 \end{cases}
\end{equation*}

Then, $u_1 \leq u_2$ in $H$.
\end{cor}

We conclude this section with the following strong maximum principle.
\begin{lem}  \label{strong-MP}
Assume that $u\in C^{2}({H})\cap C(\overline{H})$
satisfies
\begin{equation*}
\left\{
\begin{array}{ll}
D_{H,\varphi}u +c(x)u \ge 0 &\mbox{in}\; H,\\
u\ge 0 &\mbox{in}\; H,\\
u=\varphi &\mbox{on}\; \partial H,
\end{array}
\right.
\end{equation*}
where $\Omega$ is a smooth bounded domain in $\re^{n}$ and $c\in
L^{\infty}(H)$. Suppose $\varphi \geq 0$ on $\partial
H$.

Then, either $u> 0$ in $H$, or $u\equiv 0$ in $H$.
\end{lem}
\begin{proof} The proof is similar to the one of Lemma
\ref{MP2}.

Consider the harmonic extension $v$ of $u$ in $\Omega=H\times [0,+\infty)$ with lateral boundary
data $v=\varphi$ on $\partial_L \Omega$. We observe that $v\ge 0$ in
$\Omega $. Suppose that $v\not\equiv 0$ but $v=0$
somewhere in $\Omega$. Then there exists a minimum
point $x_{0}\in H$ such that $v(x_{0},0)=0$. Hence by Hopf's
lemma we see that $
\partial_\lambda v(x_{0},0)>0$. This implies that $D_{H,\varphi}u(x_{0})+c(x_{0})u(x_{0})<0$,
since $v(x_{0},0)=u(x_{0})=0$, which is a contradiction.
\end{proof}

\section{Maximal saddle solution and monotonicity properties}
In this section we prove Theorem \ref{maximal-u} concerning the existence and monotonicity properties of a maximal saddle solution. 
In the proof we will use that every saddle solution $u$ of $(-\Delta)^{1/2}u=f(u)$ is bounded above by the function $u_b(z)=\min\{1,K|u(z)|\}$ where $z=|s-t|/\sqrt 2$ is the distance to the Simons cone and $K$ is a large constant.
Let $R>0$ and consider the open region
\begin{equation}\label{TR}
T_R=\{x\in\re^{2m}: 0<t<s<R\}.
\end{equation}
Note that $T_R\supset {\mathcal O}_R={\mathcal O}\cap B_R$.

Let, as before, $v$ be the harmonic extension of a saddle solution $u$ in the half-space $\re^{2m+1}_+$.
The regularity results given in \cite{C-SM} give a uniform upper bound for $|\nabla
v|$ (see \eqref{grad}). Then, since $v=0$ on $\mathcal C\times \re^+=\{z=0\}\times \re^+$, there exists a constant $C$, depending 
only on $n$, $||u||_{\infty}$, and $||f||_{C^1}$, such that
$$|v(x,\lambda)|=|v(y,z,\lambda)|\leq C|z|, \quad \mbox{for every}\;\;(x,\lambda)\in \overline{\re_+^{2m+1}}.$$
In particular, we have that $|u(x)|=|v(x,0)|\leq C|z|$ for every $x\in \re^{2m}$.

Observe that there exists a real number $K\geq 1$ such that $$\min\{1,C|z|\}\leq \min\{1,K|u_0(z)|\}\quad \mbox{for every}\;\; z.$$
Indeed it is enough to choose
\begin{equation}\label{k} K\geq \max \{C/u'_0(0), 1/u_0(C^{-1})\}.\end{equation}
This is possible since the quantities $u'_0(0)$ and $u_0(C^{-1})$ are strictly positive.

If we choose $K$ as in \eqref{k}, then the harmonic extension $v$ in $\re^{2m+1}_+$ of every saddle solution $u$ of $\eqref{eq1-saddle}$ satisfies
\begin{equation}\label{bound-k}
 |v(x,\lambda)|\leq \min\{1,K|u_0(z)|\}\quad \mbox{for every}\;\;(x,\lambda)\in \overline{\re_+^{2m+1}}.
\end{equation}

We define
\begin{equation}\label{datobordo}
u_{b}(z):=\min\{1,K|u_0(z)|\},\end{equation}
where $K$ satisfies \eqref{k}.
Note that $u_b=0$  on $\mathcal C \cap \overline{T_R}$.

\begin{lem}\label{maxR}
Let $f$ satisfies conditions \eqref{h1}, \eqref{h2}, \eqref{h3}.

Then, there exists a positive
 solution $\usup_R$ of
$$\begin{cases}
   D_{T_R,u_b} u=f(u) & \mbox{in}\;\;T_R\\
u=u_b &\mbox{on}\;\;\partial T_R.
  \end{cases}$$ which is
maximal
 in $T_R$ in the following sense. We have that $\usup_R\geq u$ in
 $T_R$
(and hence in ${\mathcal O}_R$) for every bounded solution $u$ of $(-\Delta)^{1/2} u=f(u)$ in $\re^{2m}$ that vanishes on the Simons
cone and has the same sign as $s-t$. In addition $\usup_R$ depends
only on $s$
 and $t$.

\end{lem}

\begin{proof}
We construct a sequence of solutions of linear problems involving the operator $D_{T_R,u_b}$ and, by
the iterative use of the maximum principle, we prove that this
sequence is non increasing and it converges to the maximal
solution $\overline{u}_R$.

We set
$$Lw:=\left(D_{T_R,u_b}+a\right)w,\;\;\mbox{and}\;\;g(w):=f(w)+aw,$$
where $a$ is a positive constant chosen such that $g'(w)=f'(w)+a$
is positive for every $w$.

Next we define a sequence of functions $\usup_{R,j}$ as follows.
We set
$$\usup_{R,0}(x): =u_b= \min\{1,Ku_0(z)\},\quad\mbox{for every}\;\;x\in T_R,$$ and we define $\usup_{R,j+1}$ to be the solution of the linear
problem
\begin{equation}\label{exisusup}
\left\{
\begin{array}{rcll}
L\usup_{R,j+1} & = & g(\usup_{R,j}) &
  \text{ in } T_R\\
\usup_{R,j+1} & = & u_b &\text{ on }\;\partial{T_R}.
\end{array}
\right.
\end{equation}

Since $L$ is obtained by adding a positive constant to
$D_{T_R,u_b}$, it satisfies the maximum principles (Lemma \ref{MP2} and Corollary \ref{MP3}) and hence
the above problem admits a unique solution
$\usup_{R,j+1}=\usup_{R,j+1}(x)$. Furthermore (and here we argue
by induction), since the problem and its data are invariant by
orthogonal transformations in the first (respectively, in the
last) $m$ variables $x_i$, the solution $\usup_{R,j+1}$ depends
only on $s$ and $t$.

First, observe that by Corollary \ref{super-k}, the function $\usup_{R,0}=\min\{1,Ku_0(z)\}$ is a supersolution of problem $Lw=g(w)$, i.e., $L\usup_{R_0}\geq g(\usup_{R,0})$. This implies that
$L\usup_{R_1}=g(\usup_{R,0})\leq L\usup_{R,0}$ and then $\usup_{R,1}\leq \usup_{R,0}\leq 1$ in $T_R$. Moreover $u_b \geq
0$ on $\partial T_R$ and therefore, by Lemma \ref{MP2}, $\usup_{R,1}\geq 0$ in $T_R$.

Assume now that $0\leq \usup_{R,j}\leq \usup_{R,j-1}\leq 1$ for
some $j\geq 1$. By the choice of $a$, we have
$g(\usup_{R,j})\leq \nolinebreak g(\usup_{R,j-1})$. We get
$$L \usup_{R,j+1} = g(\usup_{R,j}) \leq g(\usup_{R,j-1})=L \usup_{R,j}.$$

Again by the maximum principle (Corollary \eqref{MP3}) $\usup_{R,j+1}\leq \usup_{R,j}$.
Besides, $\usup_{R,j+1}\geq 0$ since $g(\usup_{R,j})\geq0$.
Therefore, by induction we have proven that the sequence
$\usup_{R,j}$ is nonincreasing, that is
$$1\geq\usup_{R,0}(x)\geq \usup_{R,1}(x)\geq \dots \geq
\usup_{R,j}(x)\geq \usup_{R,j+1}(x)\geq \dots\geq 0.$$

By monotone convergence, this sequence converges to a nonnegative
solution in $T_R$, $\usup_R$, which depends only on $s$ and $t$,
and such that $\usup_R=u_b(z)$ on $\partial{T_R}$. Thus, the strong
maximum principle (Lemma \ref{strong-MP}) leads to $\usup_R>0$ in $T_R$.

Moreover, $\usup_R$
 is maximal with respect to any bounded solution $u$, $|u|<1$ in $\re^{2m}$, that vanishes on the
 Simons cone and has the same sign as $s-t$. Indeed, let $\vsup_{R,1}$ be the harmonic extension of $\usup_{R,1}$ in $T_R\times \re^+$ which is equal to $u_b$ on the lateral boundary $\partial T_R \times \re^+$. It is the solution of the following problem
\begin{equation}\label{pb-vsup}
 \begin{cases}
  \Delta \vsup_{R,1}=0 & \mbox{in}\:\:T_R\times \re^+\\
\vsup_{R,1}=u_b & \mbox{on}\:\: \partial T_R\times \re^+\\
-\partial_\lambda \vsup_{R,1}+ a \vsup_{R,1}=g(\usup_{R,0})=g(u_b) & \mbox{on} \:\:T_R\times \{0\}.
 \end{cases}
\end{equation}
Consider now $v$ the harmonic extension of $u$ in $\re^{2m+1}_+$. Then the restriction of $v$ to $T_R$, which we still call $v$, is the solution of the problem
\begin{equation}\label{pb-v}
 \begin{cases}
  \Delta v=0 & \mbox{in}\:\:T_R\times \re^+\\
 -\partial_\lambda v+ a v=g(u) & \mbox{on} \:\:T_R\times \{0\}.
 \end{cases}
\end{equation}
Recall that by \eqref{bound-k}, we have that $v\leq u_b$ in $\overline{\re_+^{2m+1}}$.
Since $g$ is increasing, then the difference $v-\vsup_{
R,1}$ is a solution of
\begin{equation}\label{v-vsup}
 \begin{cases}
  \Delta (v- \vsup_{R,1})=0 & \mbox{in}\:\:T_R\times \re^+\\
v-\vsup_{R,1}=v-u_b\leq 0 & \mbox{on}\:\: \partial T_R\times \re^+\\
-\partial_\lambda (v-\vsup_{R,1})+ a (v-\vsup_{R,1})=g(u)-g(u_b)\leq 0 & \mbox{on} \:\:T_R\times \{0\}.
 \end{cases}
\end{equation}
We claim that $v\leq \vsup_{R,1}$ in $T_R\times [0,+\infty)$. Indeed, suppose by contradiction that $v-\vsup_{R,1}$ is positive somewhere in $T_R\times[0,+\infty)$.
 Then, by the maximum principle (Lemma \ref{MP2}), $\sup(v-\vsup_{R,1})>0$ will be achieved at some point $(x_0,0)\in T_R\times\{0\}$. By Hopf's 
Lemma and since $a$ is positive, we would have
$$-\partial_\lambda (v-\vsup_{R,1})(x_0,0)+ a (v-\vsup_{R,1})(x_0,0)>0.$$
This is a contradiction with the last inequality of \eqref{v-vsup}.
Thus we have proved that $v\leq \vsup_{R,1}$ in $T_R\times \re^+$.

Suppose now that $v\leq \vsup_{R,j}$. Arguing as before, we consider the problem satisfied by $(v-\vsup_{R,j+1})$. 
Using the maximum principle and Hopf's Lemma we deduce that $v\leq \vsup_{R,j+1}$ in  $T_R\times [0,+\infty)$.
By induction, $v\leq \vsup_{R,j}$ for every $j$ and, in particular, $u\leq \usup_{R,j}$ for every $j$. Then,
$$u\leq \usup_R=\lim_{j\rightarrow\infty} \usup_{R,j}\quad \text{ in
} T_R.$$ 
\end{proof}

The following are monotonicity results for the maximal solution
constructed above.
\begin{lem}\label{monv_Rt}
Let $\usup_R$ be the function constructed in Lemma \ref{maxR}. Let
$\vsup_R$ be the harmonic function in $T_R \times (0,+\infty)$
such that $\vsup_R(x,0)=\usup_R(x)$ for every $x \in T_R$ and
$v(x,\lambda)=u_b(x)$ for every $(x,\lambda)\in
\partial T_R \times (0,+\infty)$.

Then $\partial_t \vsup_R\leq 0$.
\end{lem}

\begin{proof}
We consider the nonincreasing sequence of function $\usup_{R,j}$ constructed in the proof of Lemma \ref{maxR}. 
We set $\vsup_{R,0}(x,\lambda)=\usup_{R,0}(x)=\min\{1,Ku_0(z)\}$ for every $(x,\lambda)\in \re_+^{2m+1}$ and, for every $j\geq 1$
we call $\vsup_{R,j}$ the harmonic extension of $\usup_{R,j}$ in $T_R\times (0,+\infty)$ such that $\vsup_{R,j}(x,\lambda)=u_b(x)$ 
for every $(x,\lambda)\in \partial T_R \times (0,+\infty)$.

The function $\vsup_{R,j}$ is a solution in coordinates $s$ and
$t$ of the problem
\begin{equation*}\begin{cases}
\displaystyle \partial_{ss}\vsup_{R,j}+\partial_{tt}\vsup_{R,j}+\partial_{\lambda\lambda}\vsup_{R,j}+\frac{(m-1)}{s}
\partial_s\vsup_{R,j}+\frac{(m-1)}{t}
\partial_t\vsup_{R,j}=0& \mbox{in}\: T_R\times (0,\infty)\\
\vsup_{R,j}=u_b & \hspace{-2em} \mbox{on}\:\:\partial T_R \times
(0,+\infty),\\
-\partial_\lambda \vsup_{R,j}+ a \vsup_{R,j}=g(\vsup_{R,j-1})& \mbox{on}\:\:
T_R\times \{0\}
\end{cases}
\end{equation*}

Differentiating with respect to $t$ we get:
\begin{equation}\label{eq-t}\begin{cases}
\displaystyle -\Delta(\partial_t\vsup_{R,j})+\frac{(m-1)}{t^2}\partial_t\vsup_{R,j}=0 & \mbox{in}\:\: T_R\times (0,\infty)\\
-\partial_\lambda
(\partial_t\vsup_{R,j})+ a\partial_t\vsup_{R,j} =g'(\vsup_{R,j-1})\partial_t\vsup_{R,j-1}&
\mbox{on}\:\: T_R\times \{0\}.
\end{cases}
\end{equation}

We observe that $\partial_t\vsup_{R,j}\leq 0$ on $\partial T_R \times
(0,+\infty)$. Indeed $\vsup_{R,j}\equiv 0$ on $(\mathcal C \cap
\partial T_R)\times (0,+\infty)$ and $\vsup_{R,j}> 0$ inside
$T_R \times (0,+\infty)$. Then, $\partial_t\vsup_{R,j}\leq 0$ on $\{t=s<R\}\times (0,+\infty)$.

Moreover $\displaystyle \vsup_{R,j}=\min\{K u_0(z),1\}=\min\{K u_0((R-t)/\sqrt 2),1\}$ on $\{t<s=R\}$
 and thus $\partial_t\vsup_{R,j}=-K/\sqrt 2\dot{u}_0((R-t)/\sqrt 2)\leq 0$ on $\{t<s=R\}\times(0,+\infty)$.

Now, we argue by induction. First, recall that $$\vsup_{R,0}=\min\{K u_0(z),1\}=\min\{K u_0((s-t)/\sqrt 2),1\},$$ then $\partial_t\vsup_{R,0}\leq 0$.

Suppose that $\partial_t\vsup_{R,j-1}\leq 0$, we prove that
$\partial_t\vsup_{R,j}\leq 0$. Indeed we have that
$(m-1)/t^2\geq 0$. Moreover, for what said before, $\partial_t\vsup_{R,j}\leq 0$
on the lateral boundary of the set $T_R \times (0,+\infty)$ and it
satisfies the Neumann condition
\begin{equation}\label{neu-vt}-\partial_\lambda
(\partial_t\vsup_{R,j})+ a\partial_t\vsup_{R,j} =g'(\vsup_{R,j-1})\partial_t\vsup_{R,j-1}\quad \mbox{on}\;\,T_R\times \{0\}
.\end{equation}
Assume by contradiction that $\partial_t \vsup_{R,j}$ is positive somewhere in $T_R \times \re^+$, then, by the maximum principle the $\sup \vsup_{R,j}>0$ will be achieved at some point $(x_0,0)$ in $T_R\times \{0\}$.
Since $g'>0$ and $a>0$, applying Hopf's Lemma we get a contradiction with \eqref{neu-vt}.
This implies that $\partial_t\vsup_{R,j}\leq 0$ for every $j$ and
then, passing to the limit, that $\partial_t\vsup_{R}\leq 0$.
\end{proof}

\begin{lem}\label{monv_Ry}
Let $\usup_R$ be the function constructed in Lemma \ref{maxR}. Let
$\vsup_R$ be the harmonic function in $T_R \times (0,+\infty)$
such that $\vsup_R(x,0)=\usup_R(x)$ for every $x \in T_R$ and
$\vsup_R(x,\lambda)=u_b(x)$ for every $(x,\lambda)\in
\partial T_R \times (0,+\infty)$.

Then, $\partial_y \vsup_R\geq 0$.
\end{lem}
\begin{proof}
Consider as before the sequences of functions $\vsup_{R,j}$ and $\usup_{R,j} $.
We first observe that $\partial_y \vsup_{R,j} \geq 0$ on $\partial T_R\times (0,+\infty)$. Indeed $\vsup_{R,j}\equiv 0$ on the part of the boundary $\{t=s<R\}\times (0,+\infty)$. Thus, since $\partial_y$ is a tangential derivative here, we have $\partial_y \vsup_{R,j} \equiv 0$ on $\{t=s<R\}\times (0,+\infty)$.

Take now a point $(s=R,t,\lambda)$, with $0<t<R$, on the remaining part of the boundary.
Recall that
$\displaystyle \vsup_{R,j}\leq \usup_{R,0}= \min\{K u_0(z),1\}=\min\{K u_0((s-t)/\sqrt 2),1\}$ in all of $T_R \times(0,+\infty)$.

Then, for every $0<\delta<t$ we have
\begin{eqnarray*}\vsup_{R,j}(R-\delta,t-\delta,\lambda)&\leq& \min\left\{K u_0\left(\frac{R-\delta-(t-\delta)}{\sqrt 2}\right),1\right\}\\
 &=&\min\left\{K u_0\left(\frac{R-t}{\sqrt 2}\right),1\right\}=u_b(R,t).
\end{eqnarray*}

Then $\partial_y \vsup_{R,j}\geq 0$ on $\{t<s=R\}\times(0,+\infty)$.

Next, we consider the problem satisfied by $\partial_t \vsup_{R,j}$ and $\partial_s \vsup_{R,j}$. We recall that $\partial_t \vsup_{R,j}$ is a solution of \eqref{eq-t} and $\partial_s \vsup_{R,j}$ satisfies
\begin{equation}\label{eq-s}\begin{cases}
\displaystyle -\Delta(\partial_s\vsup_{R,j})+\frac{(m-1)}{s^2}\partial_s\vsup_{R,j}=0 & \mbox{in}\:\: T_R\times (0,\infty)\\
-\partial_\lambda
(\partial_s\vsup_{R,j})+ a \partial_s\vsup_{R,j}=g'(\vsup_{R,j-1})\partial_s\vsup_{R,j-1}&
\mbox{on}\:\: T_R\times \{0\}.
\end{cases}
\end{equation}

Thus, since $\partial_y=(\partial_s+\partial_t)/\sqrt 2$, we have that $\partial_y\vsup_{R,j}$ satisfies the equation
\begin{eqnarray*}
&& -\Delta(\partial_y\vsup_{R,j})=-\frac{m-1}{\sqrt 2}\left(\frac{\partial_s \vsup_{R,j}}{s^2}+\frac{\partial_t \vsup_{R,j}}{t^2}\right)\\
&&\hspace{4em} = -\frac{m-1}{s^2}\partial_y\vsup_{R,j}-\frac{(m-1)(s^2-t^2)}{\sqrt 2 s^2 t^2}\partial_t\vsup_{R,j}.
\end{eqnarray*}
Then $\partial_y\vsup_{R,j}$ is a solution of the problem
\begin{equation*}\begin{cases}
\displaystyle -\Delta(\partial_y\vsup_{R,j})+\frac{(m-1)}{s^2}\partial_y\vsup_{R,j}+ \frac{(m-1)(s^2-t^2)}{\sqrt 2 s^2 t^2}\partial_t\vsup_{R,j}=0& \mbox{in}\:\: T_R\times (0,\infty)\\
\partial_y \vsup_{R,j}\geq 0 &
\mbox{on}\:\: \partial T_R\times (0,+\infty)\\
-\partial_\lambda
(\partial_y\vsup_{R,j})+ a \partial_y\vsup_{R,j} =g'(\vsup_{R,j-1})\partial_y\vsup_{R,j-1}&
\mbox{on}\:\: T_R\times \{0\}.
\end{cases}
\end{equation*}
By Lemma \ref{monv_Rt} we have that $\partial_t \vsup \leq 0$ in $T_R\times (0,+\infty)$ and thus
$$
\frac{(m-1)(s^2-t^2)}{\sqrt 2 s^2 t^2}\partial_t\vsup_{R,j}\leq 0,\quad \mbox{in}\:\:T_R\times (0,+\infty).$$

Then, we can apply, as in the proof of Lemma \ref{monv_Rt}, the maximum principle and Hopf's Lemma, to obtain $\partial_y\vsup_{R,j}\geq 0$ for every $j$. Finally, passing to the limit for $j\rightarrow \infty$, we get $\partial_y \vsup_{R}\geq 0$ in $T_R\times (0,+\infty)$.
\end{proof}

We can give now the proof of Proposition \ref{maximal-u}.

\begin{proof}[Proof of Proposition \ref{maximal-u}]
In Lemma \ref{maxR} we established the existence of a maximal
solution $\usup_R$ in $T_R$, that is, $\usup_R$ is a solution of
$D_{T_R,u_b} \usup_R =f(\usup_R)$ in $T_R$ and
$$\usup_R\geq u$$ for every bounded solution $|u|\leq 1$ in $\re^{2m}$ that vanishes on
${\mathcal C}$ and has the same sign as $s-t$.

By standard elliptic estimates and the compactness arguments as in
the proof of Theorem \ref{existence}, up to a subsequence we can
take the limit as $R\rightarrow +\infty$ and obtain a solution
$\usup$ in ${\mathcal O}=\{s>t\}$, with $\usup=0$ on ${\mathcal
C}$.
 By construction, $$u\leq \usup :=\lim_{R_j\rightarrow\infty} \usup_{R_j},$$
for all solutions $u$ as above. In addition, $\usup$ depends only on
$s$ and $t$.

By maximality of $\usup$ and the existence of saddle
solution of Theorem \ref{existence}, we deduce that $\usup>0$ in
${\mathcal O}$.

Since $f$ is odd, by odd reflection with respect to the Simons cone, we obtain a maximal solution $\usup$ in $\re^{2m}$ such that $|u|\leq |\usup|$ in $\re^{2m}$.

Let $\vsup$ be the harmonic extension of $\usup$ in $\re^{2m+1}_+$.
We prove now the monotonicity properties of $\vsup$.

 By Lemmas \ref{monv_Rt} and \ref{monv_Ry}, we have that $\partial_t \vsup_{R}\leq 0$ and $\partial_y \vsup_{R}\geq 0$ in $T_R\times (0,+\infty)$.
Letting $R\rightarrow +\infty$, we get $\partial_t \vsup\leq 0$ and $\partial_y \vsup\geq 0$ in $\widetilde{\mathcal O}$. As a consequence $\partial_s \vsup \geq 0$ in $\widetilde{\mathcal O}$.

Since $v(s,t,\lambda)=-v(t,s,\lambda)$, it follows that $\partial_s \vsup\geq 0$ and $\partial_t \vsup\leq 0$ in $\re^{2m+1}_+$.

Now, $\partial_t \vsup \leq 0$ in $\re^{2m+1}_+$ and satisfies
$$-\Delta \partial_t \vsup + \frac{m-1}{t^2}\partial_t \vsup =0\quad \mbox{in}\:\:\re^{2m+1}_+.$$
Then, the strong maximum principle implies that $\partial_t \vsup <0$ in $\re^{2m+1}_+\setminus \{t=0\}.$
Moreover we multiply by $t$ the following equation satisfied by $\vsup$ in $\re^{2m+1}_+$
$$\partial_{ss}\vsup +\partial_{tt}\vsup +\partial_{\lambda \lambda}\vsup +\frac{m-1}{s}\vsup_s +\frac{m-1}{t}\vsup_t=0.$$
Using that $\vsup \in C^2$ and letting $t\rightarrow 0$, we get $ \partial_t \vsup=0$ on $\{t=0\}$.
In the same way we deduce that $\partial_s\vsup >0$ in $\re^{2m+1}_+\setminus \{s=0\}$ and $\partial_s \vsup=0$ on $\{s=0\}$.
Recalling that $\partial_z=(\partial_s-\partial_t)/\sqrt 2$, statement c) follows directly by a) and b).
Finally, we remind that $\partial_y \vsup$ satisfies
\begin{equation}
 -\Delta \partial_y \vsup = -\frac{m-1}{s^2}\partial_y \vsup -\frac{(m-1)(s^2-t^2)}{\sqrt 2 s^2 t^2}\partial_t \vsup\geq -\frac{m-1}{s^2}\partial_y \vsup,
\end{equation}
in $\{s>t>0\}\times [0,+\infty)$, since $\partial_t \vsup \leq 0$ in this set.
Since we have already proven that $\partial_y\vsup \geq 0$ in $\{s>t>0\}\times [0,+\infty)$, the strong maximum principle implies 
$\partial_y \vsup>0$ in $\{s>t>0\}\times [0,+\infty)$.
\end{proof}

%
%
%
%

\section{Asymptotic behaviour of saddle solutions in $\re^{2m}$}
In this section we study the asymptotic behaviour at infinity of solutions which are odd with respect to the Simons cone and positive in the set $\mathcal O=\{s>t\}$. 
In particular our result holds for saddle solutions.

We will consider the $(y,z)$ system of coordinates. Recall that we have defined in \eqref{defyz} $y$ and $z$ by
\begin{equation}
\begin{cases}
\displaystyle y  =   \frac{s+t}{\sqrt{2}} \\
\displaystyle z =  \frac{s-t}{\sqrt{2}},
\end{cases}
\end{equation}
which satisfy $y\geq 0$ and $-y\leq z\leq y$.

We give the proof of Theorem \ref{asym}, which states that any solution $u$ as above tends to infinity to the function
$$U(x):=u_0(z)=u_0(d(x,\mathcal C)),$$
uniformly outside compact sets. We recall that $u_0$ is the layer solution of $(-\Delta)^{1/2}u_0=f(u_0)$ in $\re$ which vanishes at the origin, and $d(\cdot, \mathcal C)$ denotes the distance to the Simons cone. Similarly $\nabla u$ converges to $\nabla U$. We will use this fact in the proof of instability of saddle solutions in dimension $2m=4$ and $2m=6$.

Our proof of the asymptotic behaviour follows a method used by Cabr\'e and Terra for the classical equation $-\Delta u=f(u)$. They use a compactness argument based on translations of the solution, combined with two crucial Liouville-type results for nonlinear equations.
Here, we use analog Liouville results for the nonlinear Neumann problem satisfied by the harmonic extension $v$ of our saddle solutions $u$. Both results were proven using the moving planes method.

The first result establishes a symmetry property for solutions of a nonlinear Neumann problem in the half-space, and it was proven in \cite{YYL}.

\begin{teo}\label{yanyanli}{(\cite{YYL})}

Let $\re^{n+1}_{+}=\{(x_{1},x_{2},\cdots,x_{n},\lambda)\mid
 \lambda > 0\}$ and let $f$ be such that $f(u)/ u^{\frac{n}{n-2}}$ is non-increasing.
 Assume that $v$ is a solution of problem
\begin{equation}\label{eqn-depend2}
\begin{cases}
-\Delta  v=0 &  \text{in}\; \re^{n+1}_{+}, \\
-\partial_\lambda v=f(v) & \text{on}\; \{ \lambda=0\}, \\
v>0  &  \text{in}\; \re^{n+1}_{+}.
\end{cases}
\end{equation}

Then $v$ depends only on $\lambda$.

More precisely, there exist $a\geq 0$ and $b>0$ such that
$$v(x,\lambda)=v(\lambda)=a\lambda+b\:\:\:
\mbox{and}\:\:\:f(b)=a.$$
\end{teo}
\begin{oss}\label{f}
If $f$ satisfies hypothesis \eqref{h1}, \eqref{h2}, \eqref{h3}, then $f(u)/ u^{\frac{n}{n-2}}$ is non-increasing.

Indeed, by Remark \ref{f(u)/u}, $f(u)/u$ is non-increasing in $(0,1)$.
Moreover, we can write
$$\frac{f(u)}{u^{\frac{n}{n-2}}}=\frac{f(u)}{u}\cdot u^{1-\frac{n}{n-2}}.$$
Since $n/n-2>1$, then $u^{1-\frac{n}{n-2}}$ is non-increasing, and thus $f$ satisfies the hypothesis of Theorem \ref{yanyanli} above and Theorem \ref{tan} below.
 \end{oss}
\begin{cor}\label{yan}
 Let $f$ satisfy \eqref{h1}, \eqref{h2}, \eqref{h3}. Let $v$ be a bounded solution of problem \eqref{eqn-depend2}.

Then, $v\equiv 0$ or $v \equiv 1$.
\end{cor}

\begin{proof}[Proof of Corollary \ref{yan}]
By Remark \ref{f}, $f$ satisfies the hypothesis of Theorem \ref{yanyanli}. Moreover since $f$ is bistable, we have that $f$ is odd, $f(0)=f(\pm 1)=0$, $f>0$ in $(0,1)$ and $f<0$ in $(1,+\infty)$.
Then, since $v$ is bounded, necessarely we have $v(x,\lambda)=b$ with $f(b)=0$, that is $v\equiv 0$ or $v\equiv 1$.
\end{proof}
The following theorem, proven in \cite{Tan}, establishes an analog symmetry property but for solutions in a quarter of space.

\begin{teo} \label{tan}{(\cite{Tan})} Let $\re^{n+1}_{++}=\{(x_{1},x_{2},\cdots,x_{n},\lambda)\mid x_{n}>0, \lambda >
0\}$ and let $f$ be such that $f(u)/ u^{\frac{n}{n-2}}$ is non-increasing.
 Assume that $v$ is a bounded solution of problem
\begin{equation*}
\begin{cases}
-\Delta  v=0 &  \text{in}\; \re^{n+1}_{++}, \\
-\partial_\lambda v=f(v) & \text{on}\; \{x_{n}>0, \lambda=0\}, \\
v=0  &\text{on}\; \{x_{n}= 0, \lambda\ge 0\},\\
 v>0 &  \text{in}\; \re^{n+1}_{++}.
\end{cases}
\end{equation*}

Then $v$ depends only on $x_{n}$
and $\lambda$.
\end{teo}
Before proving Theorem \ref{asym}, we give the following definition of semi-stability, which will be used in the proof of the asymptotic behaviour.
\begin{defin}\label{semistable}
Let $\Omega \subset \re^{n+1}_+$ be an open set. Let $v$ be a bounded solution of 
$$\begin{cases}
\Delta v=0 & \mbox{in}\;\; \Omega\\
-\partial_{\lambda}v=f(v)&\mbox{on}\;\;\partial^0 \Omega.\end{cases}$$
We say that $v$ is \emph{semi-stable} in $\Omega$ if the second variation of the energy $\delta^2\mathcal E/\delta^2 \xi^2$ with respect to perturbations $\xi$ with compact support in $\Omega \cup \partial^0 \Omega$ is nonnegative.

That is, if
$$Q_v(\xi)=\int_{\Omega}|\nabla \xi|^2dx d\lambda-\int_{\partial^0\Omega}f'(u)\xi^2dx \geq 0,$$ for all $ \xi\in C_c^{\infty}(\Omega \cup \partial^0 \Omega)$.
\end{defin} 
Now, we can give the proof of our asymptotic behaviour result.

\begin{proof}[Proof of Theorem \ref{asym}]
Let $u$ be a bounded solution of $(-\Delta)^{1/2}u=f(u)$ in $\re^{2m}$ such that $u\equiv 0$ on $\mathcal C$, $u>0$ in $\mathcal O$, and $u$ is odd with respect to $\mathcal C$.
Consider the harmonic extension $v$ of $u$
in $\re^{2m+1}_+$, that satisfies
\begin{equation}
\begin{cases}
\Delta v=0 & \mbox{in}\;\;\re^{2m+1}_+\\
-\partial_{\lambda}v=f(v) &
\mbox{on}\;\;\partial\re^{2m+1}_+\end{cases} .
\end{equation}
Set $V(x,\lambda):=v_0(z,\lambda)$. We want to prove that for every $\lambda\geq 0$
$$v(x,\lambda)-V(x,\lambda)\rightarrow 0 \quad \mbox{and }
\quad\nabla v(x,\lambda)-\nabla
 V(x,\lambda)\rightarrow 0,$$
 uniformly as $|x|\rightarrow\infty.$\\
 Suppose that the theorem does not hold. Thus, there exists $\epsilon>0$ and
a sequence $\{x_k\}$ with
\begin{equation}\label{contra}
|x_k|\rightarrow\infty\quad \text{  and }
\quad|v(x_k,\lambda)-V(x_k,\lambda)|+|\nabla v(x_k,\lambda)-\nabla
V(x_k,\lambda)|\geq \epsilon.
\end{equation}
By continuity  we may move slightly $x_k$ and assume
$x_k\not\in{\mathcal C}$ for all $k$. Moreover, up to a
subsequence (which we still denote by $\{x_k\}$), either
$\{x_k\}\subset\{s>t\}$ or  $\{x_k\}\subset\{s<t\}$. By the
symmetries of the problem we may assume
$\{x_k\}\subset\{s>t\}={\mathcal O}$.

We  distinguish  two cases:

\vspace{1em}

{\sc Case 1} $\{\mbox{ dist}(x_k,{\mathcal C})=d_k\}$ is
unbounded.

\vspace{1em}

In this case, since $0<z_k=\mbox{ dist}(x_k,{\mathcal
C})=d_k\rightarrow +\infty$ (for a subsequence), we have that
$V(x_k,\lambda)=v_0(z_k,\lambda)=v_0(d_k,\lambda)$ tends to $1$
and $|\nabla V(x_k,\lambda)|$ tends to $0$, that is,
$$V(x_k,\lambda) \rightarrow 1\quad \mbox{ and}\quad |\nabla
V(x_k,\lambda)|\rightarrow 0.$$ From this and \eqref{contra} we
have
\begin{equation}\label{case1}
|v(x_k,\lambda)-1| + |\nabla
v(x_k,\lambda)|\geq\frac{\epsilon}{2},
\end{equation}
for $k$ large enough. Taking subsequence (and relabeling the
subindex) we may assume $\mbox{dist}(x_k,{\mathcal C})=d_k\geq 2k$.

Consider the ball $B_k(0)\subset \re^{2m}$ of radius $k$ centered at
$x=0$, and define
$$w_k(\tilde{x},\lambda)=v(\tilde{x}+x_k,\lambda),\,\mbox{for every}\;(\widetilde{x},\lambda)\in B_k(0)\times (0,+\infty).$$
Since $B_k(0)+x_k \subset\{s>t\}$, we have that $0<w_k<1$ in
$B_k(0)\times (0,+\infty)$ and

\begin{equation}
\begin{cases}
\Delta w_k=0 & \mbox{in}\;\;B_k(0)\times (0,+\infty)\\
-\partial_{\lambda}w_k=f(v) &
\mbox{on}\:\:B_k(0)\times\{\lambda=0\}.\end{cases}
\end{equation}

Letting $k$ tend to infinity we obtain, through a subsequence, a
nonnegative solution $w$ of the problem 
\begin{equation}\label{solpos}
\begin{cases}
 -\Delta w=0 & \mbox{ in}\; \re^{2m+1}_+ \\
  -\partial_\lambda w=f(v) & \mbox{ on}\; \partial\re^{2m+1}_+\\
  w > 0 & \mbox{ in}\; \re^{2m+1}_+.
\end{cases}
\end{equation}

 Since
$f$ satisfies \eqref{h1}, \eqref{h2}, \eqref{h3},  we have that, by Corollary
\ref{yan}, $w\equiv 0$ or $w\equiv 1$. In either case, $\nabla
w(0)=0$, that is, $|\nabla v(x_k,\lambda)|$ tends to $0$.

Next we show that $w\not\equiv 0$. By Theorem \ref{existence} we have
that $v$ is stable in ${\mathcal O}\times (0,+\infty)$. Hence,
$w_k$ is semi-stable in $B_k(0)\times (0,+\infty)$ (since
$B_k(0)+x_k\subset {\mathcal O}$) in the sense of Definition \ref{semistable}. This implies that $w$ is stable in
all of $\re^{2m+1}_+$ and therefore $w\not\equiv 0$ (otherwise,
since $f'(0)>0$ we could construct a test function $\xi$ such that
$Q_w(\xi)<0$ which would be a contradiction with the fact that $w$
is stable).

Hence, it must be $w \equiv 1$. But this implies that
$w(0,\lambda)=1$ and so $v(x_k,\lambda)$ tends to $1$. Therefore, we
have that $v(x_k,\lambda)$ tends to $1$ and $|\nabla
v(x_k,\lambda)|$ tends to $0$, which is a contradiction with
($\ref{case1}$). We have proven the theorem in this
case 1.

\vspace{1em}

{\sc Case 2} $\{\mbox{ dist}(x_k,{\mathcal C})=d_k\}$ is bounded.

\vspace{1em}

The points $x_k$ remain at a finite distance to the cone. Then, at
least for a subsequence,
$$d_k\rightarrow d\geq 0 \quad \mbox{ as}\; k\rightarrow\infty.$$
Let $x_k^0\in{\mathcal C}$ be a point that realizes the distance
to the cone, that is,
\begin{equation}\label{case2}
\mbox{ dist}(x_k,{\mathcal C})=|x_k-x_k^0|=d_k,
\end{equation}
and let $\nu_k^0$ be the inner unit normal to ${\mathcal
  C}=\partial{\mathcal O}$ at $x_k^0$. Note
that $B_{d_k}(x_k)\subset{\mathcal
O}\subset\re^{2m}\setminus{\mathcal C}$ and $x_k^0\in\partial
B_{d_k}(x_k)\cap{\mathcal C}$, i.e., $x_k^0$ is the point where
the sphere $\partial B_{d_k}(x_k)$ is tangent to the cone
${\mathcal C}$. It follows that $x_k^0\neq 0$ and that
$(x_k-x_k^0)/d_k$ is the unit normal $\nu_k^0$ to ${\mathcal C}$
at $x_k^0$. That is, $x_k=x_k^0+d_k\nu_k^0$.

Now, since the sequence $\{\nu_k^0\}$ is bounded, there exists a
subsequence such that
$$\nu_k^0\rightarrow \nu \in\re^{2m}, \quad |\nu|=1.$$

Write $w_k(\tilde{x},\lambda)=v(\tilde{x}+x_k^0,\lambda)$, for
$\tilde{x}\in\re^{2m}$. The functions $w_k$ are all solutions of
\begin{equation}
\begin{cases}
\Delta w_k=0 & \mbox{in}\; \re^{2m+1}_+\\
-\partial_{\lambda}w_k=f(w_k)&
\mbox{on}\;\partial\re^{2m+1}_+.\end{cases}
\end{equation}
 and are uniformly bounded.
Hence, by elliptic estimates, the sequence $\{w_k\}$ converges
locally in space in $C^2$, up to a subsequence, to a solution $w$
in $\re^{2m+1}_+$. Therefore we have that, as $k$ tends to
infinity and up to a subsequence,
$$w_k\rightarrow w \quad \mbox{ and}\quad \nabla w_k\rightarrow \nabla w
\;\text{uniformly on compact sets of } \re^{2m+1}_+,$$ where $w$
is a solution
\begin{equation}
\begin{cases}
\Delta w=0 & \mbox{in}\; \re^{2m+1}_+\\
-\partial_{\lambda}w=f(w)& \mbox{on}\;\partial\re^{2m+1}_+.\end{cases}
\end{equation} Note that the curvature of ${\mathcal C}$ at
$x_k^0$ goes to zero as $k$ tends to infinity, since ${\mathcal
C}$ is a cone and $|x_k|\rightarrow\infty$ (note that
$|x_k^0|\rightarrow\infty$ due to $|x_k|\rightarrow\infty$ and
$|x_k-x_k^0|=d_k\rightarrow d<\infty$). Thus, ${\mathcal C}$ at
$x_k^0$ is flatter and flatter as $k\rightarrow\infty$ and since
we translate $x_k^0$ to $0$, the limiting solution $w$ satisfies

\begin{equation}
\begin{cases}
\Delta w=0 & \mbox{in}\;\;M:=\{(x,\lambda)\in
\re^{2m+1}_+:\widetilde{x}\cdot \nu=0, \lambda>0\}\\
w\geq 0 & \mbox{in}\;M\\
w=0 & \mbox{on}\;\{\widetilde{x}\cdot \nu=0\}\\
-\partial_\lambda w=f(w)&\mbox{on}\;\{\lambda=0\}.
 \end{cases}
 \end{equation}
For the details of the proof of this fact see \cite{CT2}.

Now, since $v$ is stable for perturbations vanishing on $\partial
{\mathcal O}\times\re^+$, it follows that $w$ is stable for
perturbations with compact support in $M$, and therefore $w$ can
not be identically zero. By Theorem \ref{tan}, since $f$ satisfies \eqref{h1}, \eqref{h2}, \eqref{h3}, we deduce
 that $w$ is symmetric, that is, it is a function
of only two variable (the orthogonal direction to $H$ and
$\lambda$). It follows that
\[
 w (\tilde x,\lambda) =
v_0(\tilde x \cdot \nu,\lambda)\quad \text{ for all
}(\tilde{x},\lambda)\in M.
\]
From the definition of $w_k$, and using that $z_k=d_k=|x_k-x_k^0|$
is a bounded sequence and that $x_k-x_k^0=d_k\nu_k^0$, we have
that
\begin{eqnarray*}
v(x_k,\lambda) & = & w_k(x_k -x_k^0,\lambda) = w(x_k -
x_k^0,\lambda) +\mbox{ o}(1) =
v_0 ((x_k - x_k^0) \cdot \nu,\lambda) +\mbox{ o}(1)  \\
 & = &  v_0((x_k - x_k ^0)\cdot
\nu_k^0,\lambda)+\mbox{ o}(1) = v_0(d_k,\lambda) +\mbox{ o}(1) \\
& = &  v_0(z_k,\lambda)+\mbox{ o}(1)= V(x_k,\lambda)+\mbox{ o}(1).
\end{eqnarray*}
The same argument can be done for $\nabla v(x_k,\lambda)$ and
$\nabla V(x_k,\lambda)$. We arrive to a contradiction with
\eqref{contra}.
\end{proof}

\section{Instability in dimensions 4 and 6}
Before proving the theorem on the instability of
saddle solutions in dimensions 4 and 6, we establish a lemma that
will be useful later.

\begin{lem}\label{lemma-ins}
Assume that $f$ satisfies conditions \eqref{h1}, \eqref{h2}, \eqref{h3}.
Let $v$ be a bounded solution of (\ref{eq2-saddle}) in $\re^{n+1}_+$ and
$w$ a function such that $|v|\leq |w|\leq 1$ in $\re^{n+1}_+$.
Then,
$$Q_v(\xi)\leq Q_w(\xi)\quad \text{ for all }
\xi\in C_0^{\infty}(\overline{\re_+^{n+1}}),$$ where $Q_w$ is
defined by
\[
Q_w(\xi)=\int_{\re^{n+1}_+} |\nabla\xi|^2dx
d\lambda-\int_{\partial \re^{n+1}_+}f'(w)\xi^2dx.
\]

In particular, if there exists a function $\xi\in
C_0^{\infty}(\overline{\re^{n+1}_+})$ such that $Q_w(\xi)<0$, then
$v$ is unstable.
\end{lem}

\begin{proof}
Let $v$ be a bounded solution of \eqref{eq2-saddle} and $w$ a function with
$|v|\leq |w|\leq 1$.

Since $f'$ is decreasing in $(0,1)$ we have that
$$f'(|v|)\geq f'(|w|)\quad \text{ in }\re^{n+1}_+.$$ Moreover,  $f'$ being even
yields,
$$f'(v)\geq f'(w)\quad \text{ in }\re^{n+1}_+,$$
so that $$Q_v(\xi)\leq Q_w(\xi),$$ for every test function $\xi\in
C_0^{\infty}(\overline{\re_+^{n+1}})$.

Hence, if there exists $\xi_0$ such that $Q_w(\xi_0)<0$, then also
$Q_v(\xi_0)<0$. That is, $v$ is unstable.
\end{proof}

In the proof of the instability results for dimension $4$ and $6$
we use the maximal solution $\vsup$ of problem (\ref{eq2-saddle})
and, more importantly, the equation satisfied by
$\vsup_z=\partial_z\vsup$. We prove that this solution $\vsup$ is
unstable by constructing a test function
$\xi(y,z,\lambda)=\eta(y,\lambda)\vsup_z(y,z,\lambda)$ such that
$Q_{\vsup}(\xi)<0$. Two crucial ingredients will be the asymptotic behaviour and
monotonicity results for $\vsup$ (Theorems \ref{asym} and
\ref{maximal-u}). Since $\vsup$ is maximal,  Lemma \ref{lemma-ins}
implies that all bounded solutions $-1\leq v\leq 1$ vanishing on
${\mathcal C}\times \re^+$ and having the same sign as $s-t$ are
also unstable.

We recall that if $v$ is a function depending only on $s$, $t$ and $\lambda$, then the second variation of the energy is given by
\begin{eqnarray*}c_m Q_v(\xi)&=&\int_0^{+\infty}\int_{\{s>0,\:t>0\}}s^{m-1}t^{m-1}(\xi_s^2+\xi_t^2+\xi_\lambda^2)ds dt d\lambda\\
&& -\int_{\{s>0,\:t>0\}}s^{m-1}t^{m-1}f'(v)\xi^2 ds dt,\end{eqnarray*}
where $c_m$ is a positive constant depending on $m$. Here, the perturbations are of the form $\xi=\xi(s,t,\lambda)$ and vanishes for $s,t$ and $\lambda$ large enough.

Moreover, if we change to variables $(y,z,\lambda)$, for a different constant $c_m$ we get,
\begin{eqnarray*}c_m Q_v(\xi)&=&\int_0^{+\infty}\int_{\{-y<z<y\}}(y^2-z^2)^{m-1}(\xi_y^2+\xi_z^2+\xi_\lambda^2)dy dz d\lambda \\
&&-\int_{\{-y<z<y\}}(y^2-z^2)^{m-1}f'(v)\xi^2 dy dz,\end{eqnarray*}
where $\xi=\xi(y,z,\lambda)$ vanishes for $y$ and $\lambda$ large enough.

\begin{proof}[Proof of Theorem $\ref{uns6}$]  We begin by
establishing that the maximal solution $\vsup$ is unstable in dimension $2m=4$ and $2m=6$. As said before, using that $\vsup$ is maximal and applying Lemma \ref{lemma-ins}, we deduce the instability of $v$ in dimensions $4$ and $6$. 

We have, for every test function $\xi$,
$$Q_{\vsup}(\xi)=\int_{\re_+^{2m+1}}|\nabla
\xi|^2dx d\lambda-\int_{\partial\re^{2m+1}_+}f'(\vsup)\xi^2dx.$$
Suppose now that $\xi=\xi(y,z,\lambda)=\eta(y, z,\lambda)\psi(y,
z,\lambda)$. For $\xi$ to be Lipschitz and of compact support in
$\overline{\re_+^{2m+1}}$, we need $\eta$ and $\psi$ to be
Lipschitz functions of compact support in $y\in[0,+\infty)$ and
$\lambda\in [0,+\infty)$. The expression for $Q_{\vsup}$ becomes,
\begin{eqnarray*}
&&Q_{\vsup}(\xi)=\int_0^{+\infty}\int_{\re^{2m}}\left(|\nabla
\eta|^2\psi^2+\eta^2|\nabla\psi|^2 +
2\eta\psi\nabla\eta\cdot\nabla\psi\right)dx d\lambda \\
&&\hspace{6em} -\int_{
\re^{2m}}f'(\vsup)\eta^2\psi^2dx.\end{eqnarray*} Using that $
2\eta\psi\nabla\eta\cdot\nabla\psi=\psi\nabla(\eta^2)\cdot\nabla\psi,$
and integrating by parts this term we have
\begin{eqnarray*}
&&Q_{\vsup}(\xi)=\int_0^{+\infty}\int_{\re^{2m}}\left(|\nabla
\eta|^2\psi^2- \eta^2\psi\Delta\psi\right)dx
d\lambda\\
&&\hspace{4em} -\int_{\re^{2m}}\left(\psi(y,z,0)
\eta^2\partial_\lambda
\psi(y,z,0)+f'(\vsup)\eta^2\psi^2\right)dx,\end{eqnarray*} that
is,
$$Q_{\vsup}(\xi)=\int_0^{+\infty}\int_{\re^{2m}}\left(|\nabla
\eta|^2\psi^2- \eta^2\psi\Delta\psi\right)dx
d\lambda-\int_{\re^{2m}}\eta^2 \psi
(\partial_{\lambda}\psi+f'(\vsup)\psi)dx.$$

Choose $\psi(y,z,\lambda)=\vsup_z(y,z,\lambda)$.
We consider now problem (\ref{eq2-saddle}), which is satisfied by $\vsup$, written in the $(y,z,\lambda)$ variables
\begin{equation}\begin{cases}
\displaystyle \vsup_{yy}+\vsup_{zz}+\vsup_{\lambda\lambda}+\frac{2(m-1)}{y^2-z^2}(y\vsup_y-z\vsup_z)=0 & \mbox{in}\:\:\re^{2m+1}_+\\
-\partial_\lambda \vsup=f(\vsup) & \mbox{on}\:\:
\partial{\re^{2m+1}_+}.
\end{cases}
\end{equation}

If we differentiate these
equations written in $(y,z,\lambda)$ variables
with respect to $z$, we find
\begin{equation}\begin{cases}
\displaystyle \Delta \vsup_z-\frac{2(m-1)}{y^2-z^2}\vsup_z+\frac{4(m-1)z}{(y^2-z^2)^2}\left(y\vsup_y-z\vsup_z\right)=0 & \mbox{in}\:\:\re^{2m+1}_+\\
-\partial_\lambda \vsup_z=f'(\vsup)\vsup_z & \mbox{on}\:\:
\partial{\re^{2m+1}_+}.
\end{cases}
\end{equation}
Replacing in the expression for $Q_{\vsup}$ we obtain,
\[
\hspace{-15em}Q_{\vsup}(\xi)=\int_0^{+\infty}\int_{\re^{2m}}{\Big
(}|\nabla\eta|^2{\vsup}_z ^2-
\]
\[
\hspace{10em}-\eta^2{\Big
\{}\frac{2(m-1)(y^2+z^2)}{(y^2-z^2)^2}{\vsup}_z^2 -
\frac{4(m-1)zy}{(y^2-z^2)^2}{\vsup}_y{\vsup}_z{\Big \}}{\Big
)}dxd\lambda.
\]
Next we change coordinates to $(y,z,\lambda)$ and we have, for
some positive constant $c_m$,
\[
c_mQ_{\vsup}(\xi)=\int_0^{+\infty}\int_{\{-y<z<y\}}(y^2-z^2)^{m-1}{\Big
(}|\nabla\eta|^2{\vsup}_z ^2-
\]
\[
\hspace{8em}-\eta^2{\Big
\{}\frac{2(m-1)(y^2+z^2)}{(y^2-z^2)^2}{\vsup}_z^2 -
\frac{4(m-1)zy}{(y^2-z^2)^2}{\vsup}_y{\vsup}_z{\Big \}}{\Big
)}dydzd\lambda.
\]

Now choose $\eta(y,z,\lambda)=\eta_1(y)\eta_2(\lambda)$, where
$\eta_1$ and $\eta_2$ are smooth functions with compact support in
$[0,+\infty)$. Moreover we require that $\eta_2(\lambda)\equiv 1$ for $\lambda <N$ and $\eta_2(\lambda)\equiv 0$ for $\lambda>N+1$, where $N$ is a large positive number that we will choose later. For $a>1$, a constant that we will make tend to
infinity, let $\phi=\phi(\rho)$ be a Lipschitz function of
$\rho:=y/a$ with compact support $[\rho_1,\rho_2]\subset
[0,+\infty)$. Let us denote by
$$\eta_1^a(y):=\phi(y/a)\quad \text{ and }$$
$$
\xi_a(y,z,\lambda)=\eta_1^a(y)\eta_2(\lambda)\vsup_z(y,z,\lambda)=\phi(y/a)\eta_2(\lambda)\vsup_z(y,z,\lambda).$$
The change
$y=a\rho,dy=ad\rho$ yields,
\[{
c_mQ_{\vsup}(\xi_a)=a^{2m-3}\int_0^{N+1}\int_{\{-a\rho
<z<a\rho
  \}}\rho^{2(m-1)}\left(1-
\frac{z^2}{a^2\rho^2}\right)^{m-1}{\Big
(}\phi_{\rho}^2\eta_2^2(\lambda){\vsup}_z^2}
\]\begin{equation}\label{dim_m}
+a^2\phi^2(\rho)(\eta'_2)^2\vsup_z^2-\phi^2\eta_2^2{\Big
\{}\frac{2(m-1)(1+\frac{z^2}{a^2\rho^2})}{\rho^2(1-\frac{z^2}{a^2\rho^2})^2}{\vsup}_z^2
-
\frac{4(m-1)z}{a\rho^3(1-\frac{z^2}{a^2\rho^2})^2}{\vsup}_y{\vsup}_z{\Big
\}}{\Big )}d\rho dz.
\end{equation}

Dividing by $a^{2m-3}N$ and using that $\left(1-\frac{z^2}{a^2\rho^2}\right)^2\leq 1$
and $1+\frac{z^2}{a^2\rho^2}\geq 1$, we obtain
\begin{eqnarray*}
&&\hspace{-0.3em}\frac{c_mQ_{\usup}(\xi_a)}{a^{2m-3}N}\leq\nonumber \\
&&\hspace{0.6em} \leq\frac{1}{N} \int_0^{N+1}
\int_{\{-a\rho <z<a\rho \}}
\rho^{2(m-1)}\eta_2^2\vsup_z^2(a\rho,z,\lambda)\left(\phi_{\rho}^2-\frac{2(m-1)}{\rho^2}\phi^2\right)d\rho
dz d\lambda  \label{sum1} \\
&&\hspace{1.3em}+\frac{a^2}{N}\int_N^{N+1}\int_{\{-a\rho <z<a\rho
\}}\rho^{2(m-1)}\phi^2(\eta'_2)^2\vsup_z^2d\rho dz d\lambda\label{sum2} \\
&&\hspace{1.3em}+\frac{1}{N} \int_0^{N+1} \int_{\{-a\rho <z<a\rho \}}
\frac{4(m-1)z\rho^{2m-5} \eta_2^2\phi^2(\rho)}{a}{\vsup}_y(a\rho,z,\lambda){\vsup}_z(a\rho,z,\lambda)
d\rho dzd\lambda.\label{sum3}\\
&&\hspace{0.6em}=I_1+I_2+I_3.
\end{eqnarray*}

We study these three integrals separately.

Consider first $I_3$. From Theorem
$\ref{asym}$ we have that $\vsup_y(a\rho,z,\lambda)\rightarrow 0$
uniformly, for all $\rho\in[\rho_1,\rho_2]=\mbox{supp}\phi$, as
$a$ tends to infinity. Hence, given $\epsilon>0$, for $a$
sufficiently large, $|\vsup_y(a\rho,z)|\leq \epsilon$. Moreover,
we have seen in Theorem \ref{maximal-u} that $\vsup_z\geq 0$. Hence,
since $\phi$ is bounded, for $a$ large we have
\begin{eqnarray*}
I_3&\leq& \left|\frac{1}{N}\int_0^{N+1}\eta_2^2\int
\frac{4(m-1)z\rho^{2m-5}\phi^2(\rho)}{a}{\vsup}_y{\vsup}_z d\rho dz
d\lambda\right|  \\
&\leq&\frac{1}{N}\int_0^{N+1}\eta_2^2\int \left|\frac{4(m-1)z\rho^{2m-5}\phi^2(\rho)}{a}\right||{\vsup}_y|{\vsup}_z d\rho dz d\lambda \\
&\leq& \frac{1}{N}\int_0^{N+1}\eta_2^2\int 4(m-1)\rho^{2m-4}\phi^2(\rho)|\vsup_y|\vsup_z d\rho dz d\lambda\\
&\leq& \frac{C\epsilon}{N} \int_{\rho_1}^{\rho_2} \rho^{2m-4}d\rho\int_0^{N+1}\eta_2^2d\lambda \int_{-a\rho}^{a\rho}\vsup_zdz \\
&=& \frac{C\epsilon}{N} \int_0^{N+1}\eta_2^2\int_{\rho_1}^{\rho_2} \left(\vsup(a\rho,a\rho,\lambda)-\vsup(a\rho,-a\rho,\lambda)\right) d\rho d\lambda\\
&\leq& C\epsilon,
\end{eqnarray*}
where $C$ are different constants depending on $\rho_1$ and $\rho_2$. Hence, as $a$ tends to infinity, this integral converges
to zero.

Now, consider $I_2$ and choose $N$ such that $a^2/N\leq 1/a^2$. With this choice of $N$, we have
\begin{eqnarray*}&&I_2=\frac{a^2}{N}\int_N^{N+1}\int_{\rho_1}^{\rho_2}\int_{\{-a\rho <z<a\rho
\}}\rho^{2(m-1)}\phi^2(\eta'_2)^2\vsup_z^2 \\
&&\hspace{2em}\leq \frac{1}{a^2}\int_N^{N+1}\int_{\rho_1}^{\rho_2}\int_{\{-a\rho <z<a\rho
\}}\rho^{2(m-1)}\phi^2(\eta'_2)^2\vsup_z^2\\
&&\hspace{2em}\leq \frac{C}{a}\sup\vsup_z^2.\end{eqnarray*}
Thus, $I_2$ tends to $0$ as $a\rightarrow \infty$.

Next, consider $I_1$. We have that, again by Theorem
\ref{asym}, $\vsup_z(a\rho,z,\lambda)$ converges to
$\partial_z{v_0}(z,\lambda)$ which is a bounded positive
integrable function. We write
\begin{eqnarray*}
I_1&=&\frac{1}{N}\int_0^{N+1}\eta_2^2\int_{\{-a\rho <z<a\rho \}}
\rho^{2(m-1)}\vsup_z^2(a\rho,z,\lambda)\left(\phi_{\rho}^2-\frac{2(m-1)}{\rho^2}\phi^2\right)d\rho
dz d\lambda= \\
&=&\frac{1}{N}\int_0^{N+1}\eta_2^2\int_{\{-a\rho <z<a\rho \}} (\partial_z{v}_0)^2 \rho^{2(m-1)}\left(\phi_{\rho}^2-\frac{2(m-1)}{\rho^2}\phi^2\right)d\rho dz d\lambda \\
&&\hspace{0.2em} + \frac{1}{N}\int_0^{N+1}\eta_2^2\int_{\{-a\rho <z<a\rho
\}}
\rho^{2(m-1)}(\vsup_z(a\rho,z,\lambda)-\partial_zv_0(z,\lambda))(\vsup_z(a\rho,z,\lambda)\nonumber
\\&&\hspace{0.2em} +
\partial_zv_0(z,\lambda))\left(\phi_{\rho}^2-\frac{2(m-1)}{\rho^2}\phi^2\right)d\rho
dzd\lambda.
\end{eqnarray*}
For $a$ large,
$|\vsup_z(a\rho,z,\lambda)-\partial_z{v}_0(z,\lambda)|\leq \epsilon$
in $[\rho_1,\rho_2]$. In addition
$\vsup_z(a\rho,z,\lambda)+\partial_zv_0(z,\lambda)$ is positive
and is a derivative with respect to $z$ of a bounded function,
thus it is integrable in $z$. Hence, since $\phi=\phi(\rho)$ is
smooth with compact support, the second integral converges to zero
as $a$ tends to infinity.
Therefore, letting $a$ tend to infinity,
we obtain
\begin{eqnarray}\label{limsup}
&&\limsup_{a\rightarrow\infty}\frac{c_mQ_{\vsup}(\xi_a)}{a^{2m-3}N}\leq\\
&&\hspace{0.2em}\leq\limsup_{a\rightarrow\infty}\frac{1}{N}\left(\int_0^{N+1}d\lambda\: \eta_2^2\int_{0}^{+\infty}dz\: (\partial_z v_0)^2(z)
\right) \int d\rho\:
\rho^{2(m-1)}\left(\phi_{\rho}^2-\frac{2(m-1)}{\rho^2}\phi^2\right) \nonumber \\
&&\hspace{0.4em}\leq C\int_{0}^{+\infty} (\partial_z v_0)^2(z)
dz\int
\rho^{2(m-1)}\left(\phi_{\rho}^2-\frac{2(m-1)}{\rho^2}\phi^2\right)d\rho.
\nonumber
\end{eqnarray}
%

Finally, we prove that when $2m=4$ and $2m=6$, there exists a test
function $\phi$ for which
\begin{equation}\label{dim6}\int
\rho^{2(m-1)}\left(\phi_\rho^2-\frac{2(m-1)}{\rho^2}\phi^2\right)d\rho<0.\end{equation}
The integral in $\rho$ can be seen as an integral in
$\re^{2m-1}$ of radial functions $\phi=\phi(|x|)=\phi(\rho)$.

Using
Hardy's inequality we have that the integral in \eqref{dim6} is positive for all Lipschitz
$\phi$ with compact support if and
only if
$$2(m-1)\leq \frac{(2m-1-2)^2}{4}.$$
Writing $n=2m$, the above inequality holds if and only if
$$n^2-10n+17\geq 0,$$
that is,  $n\geq 8$. Thus, when $2m=4$ and $2m=6$, we have that the integral
\eqref{dim6} is negative for some compactly supported Lipschitz
function $\phi=\phi(\rho)$ and then we
conclude that the limsup in \eqref{limsup} is negative for such $\phi$
and hence that $\usup$ is unstable.
\end{proof}
\begin{oss}
We observe that for $n\geq 8$ the limsup in \eqref{limsup} is nonnegative for every $\phi$ and we conclude a certain asymptotic stability at infinity of $\vsup$.
\end{oss}

\end{document}